\documentclass[]{sn-jnl}

\usepackage{lineno}
\usepackage[UKenglish]{babel}
\usepackage{amsmath, amsthm, amssymb, dsfont, physics, hyperref, cleveref, graphicx, graphicx, amsfonts}

\newtheorem{theorem}{Theorem}
\newtheorem{definition}[theorem]{Definition}
\newtheorem{proposition}[theorem]{Proposition}
\newtheorem{lemma}[theorem]{Lemma}
\newtheorem{corollary}[theorem]{Corollary}

\theoremstyle{definition}
\newtheorem{example}[theorem]{Example}
\newtheorem{remark}[theorem]{Remark}

\crefname{definition}{Definition}{Definition}
\crefname{theorem}{Theorem}{Theorem}
\crefname{lemma}{Lemma}{Lemma}
\crefname{proposition}{Proposition}{Proposition}
\crefname{corollary}{Corollary}{Corollary}
\crefname{equation}{}{Eq.}
\Crefname{equation}{Eq.}{Eq.}

\newcommand{\paren}[1]{\left(#1\right)}

\renewcommand{\x}{{x}}
\newcommand{\xM}{{X}}
\newcommand{\X}{\xM}
\newcommand{\xT}{{\mathpzc{X}}}
\newcommand{\calX}{{\mathcal{X}}}
\renewcommand{\y}{{y}}
\newcommand{\yM}{{Y}}
\newcommand{\Y}{\yM}

\newcommand{\calY}{{\mathcal{Y}}}
\newcommand{\tuck}{G_{\mathcal{T}}}

\newcommand{\lscond}{\kappa_{\star}}

\DeclareMathAlphabet{\mathpzc}{OT1}{pzc}{m}{it}

\newif\ifdraft
\drafttrue

\title{Which constraints of a numerical problem cause ill-conditioning?}

\author*[1]{\fnm{Nick} \sur{Dewaele}} \email{nick.dewaele@kuleuven.be}
\author[1,2]{\fnm{Nick} \sur{Vannieuwenhoven}} \email{nick.vannieuwenhoven@kuleuven.be}

\affil[1]{KU Leuven, Department of Computer Science, Celestijnenlaan 200A, box 2402, B-3001 Leuven, Belgium}
\affil[2]{Leuven.AI, KU Leuven Institute for AI, B-3000 Leuven, Belgium}

\artnote{This research was funded by Interne Fondsen KU Leuven/Internal Funds KU Leuven with project number C16/21/00 and by the Research Foundation-Flanders (FWO) with project number G080822N. It was partially supported by the Thematic Research Programme ``Tensors: geometry, complexity and quantum entanglement,'' University of Warsaw, Excellence Initiative -- Research University and the Simons Foundation Award No. 663281 granted to the Institute of Mathematics of the Polish Academy of Sciences for the years 2021--2023.}



\counterwithin{theorem}{section}
\counterwithin{figure}{section}
\counterwithin{table}{section}
\counterwithin{equation}{section}

\begin{document}

\abstract{
    Many numerical problems with input $\x$ and output $\y$ can be formulated as a system of equations $F(\x, \y) = 0$ where the goal is to solve for $\y$. The condition number measures the change of $\y$ for small perturbations to $\x$. From this numerical problem, one can derive a (typically underdetermined) relaxation by omitting any number of equations from $F$. We propose a condition number for underdetermined systems that relates the condition number of a numerical problem to those of its relaxations, thereby detecting the ill-conditioned constraints.
    We illustrate the use of our technique by computing the condition of two problems that do not have a finite condition number in the classic sense: two-factor matrix decompositions and Tucker decompositions. 
}

\maketitle

\ifdraft

\section{Introduction}
\label{sec: intro}

Any computational problem with input $x \in \mathcal X$ and output $y \in \mathcal Y$ can be characterised by defining a set $\mathcal{P} \subseteq \mathcal X \times \mathcal Y $ containing all admissible input-output pairs $(x,y) \in \mathcal{X} \times \mathcal{Y}$. From here on, we refer to $\mathcal{P}$ as the problem.
% This set is sometimes called the \emph{solution variety}, as in \cite[Chapter 14]{Burgisser2013}, but we refer to it simply as \emph{the problem $\mathcal P$}.
Given two problems $\mathcal{P}$ and $\mathcal{P}'$, we say that $\mathcal{P}'$ is \emph{less constrained} than $\mathcal{P}$ or a \emph{relaxation of $\mathcal P$} if $\mathcal{P} \subseteq \mathcal{P}'$. The following are typical examples of numerical problems and relaxations.
\begin{itemize}
    \item A problem defined by a system of equations $F(\x,\y) = 0$ can be relaxed by removing any number of equations.
    \item In many applications, a matrix $X \in \mathbb{R}^{m \times n}$ of a known low rank $k$ is decomposed as a product $X = LR$ where $L \in \mathbb{R}^{m \times k}$ and $R \in \mathbb{R}^{k \times n}$. In practice, this problem is usually made more constrained by imposing structure on the tuple $Y = (L,R)$, such as nonnegativity, orthogonality of the columns of $L$, or by imposing that $L$ contains a subset of the columns of $X$ \cite{trefethenNumericalLinearAlgebra1997,mahoneyCURMatrixDecompositions2009}. Such decompositions are especially preferred for large matrices of low rank \cite{halkoFindingStructureRandomness2011}.
    \item For a matrix $\xM \in \mathbb{R}^{m \times n}$ of rank $k$, computing an \emph{(arbitrary)} basis of the column space is less constrained than computing \emph{(specifically)} the first $k$ columns of $U$ in a singular value decomposition $X = U\Sigma V^T$.
    \item A \emph{Tucker decomposition} \cite{tucker1966some} of a tensor $\xT \in \mathbb{R}^{n_1 \times \dots \times n_D}$ of multilinear rank $(k_1, \dots, k_D)$ is a relaxation of the \emph{higher-order singular value decomposition} \cite{tucker1966some,DeLathauwer2000}. Similarly, computing a \emph{Tensor train} decomposition is a relaxation of the TT-SVD problem \cite{Oseledets2011}.
\end{itemize}

We say that a problem $\mathcal{P}$ is \emph{identifiable} if a unique $(x,y) \in \mathcal{P}$ exists for every input $x$. If this $y$ is a continuous function of $x$, then $\mathcal{P}$ is \emph{well-posed}. 
In numerical analysis, well-posed problems have a \emph{condition number}, which measures the local sensitivity of the output with respect to small changes in the input. The goal of this paper is to understand the condition number of $\mathcal{P}$ in relation to the condition number of any relaxation $\mathcal{P}'$ of $\mathcal{P}$.
The main obstacle is that \emph{$\mathcal{P}'$ may be too unconstrained to be well-posed and have a condition number in the usual sense}.

To elaborate the concept, consider this issue in the context of linear systems.
For $m \le n$, equip  $\mathbb{R}^m$ and $\mathbb{R}^n$ with an inner product and let $A \in \mathbb{R}^{m \times n}$ be a fixed matrix of rank $m$.
Consider the system $Ay = x$ for some input $x \in \mathbb{R}^m$. The sensitivity of ``solving for $y$'' can be interpreted in several ways.
\begin{enumerate}
    \item[0.] Since the problem is not identifiable, its forward error and condition number are left undefined.
    \item \label{item: minimise solution norm} We may add constraints to the system to obtain a unique solution.
    A common way to do this is to minimise $\norm{y}$ over all $y$ such that $Ay = x$ \cite[Section 5.5]{Golub2013}. This more constrained problem is well-posed if $A$ is fixed and can be solved by $y = A^\dagger x$ where $A^\dagger$ is the Moore--Penrose inverse of $A$.
    % The (absolute) condition number of evaluating $x \mapsto A^\dagger x$ with respect to normwise perturbations of $x$ is $\norm{A^\dagger}$, where $\norm{A^\dagger}$ is the operator norm of $A$ \cite[eq. 12.2]{trefethenNumericalLinearAlgebra1997}. 

    \item \label{item: Hausdorff distance}
    We write the solution corresponding to $x$ as a set $S_x := \left\{y\,\middle\vert\, Ay = x\right\} \subseteq \mathbb{R}^n$.
    % \begin{equation}
    %     \label{eq: Sx linear system}
    % S_x := \left\{y \,\middle\vert\, Ay = x \right\} = \left\{A^\dagger x + k \,\middle\vert\, k \in \ker A\right\}
    % \end{equation} 
    % where $S_x \subseteq \mathbb{R}^n$.
    In the language of \cite{dontchevImplicitFunctionsSolution2014}, the map $x \mapsto S_x$ is a \emph{set-valued solution map} and the forward error can be quantified in the \emph{Pompeiu--Hausdorff distance} $d_{PH}$. For two subsets $S, \tilde{S} \subseteq \mathbb{R}^n$, this is defined as 
    \begin{linenomath*}
    $$
    d_{PH}(S, \tilde{S}) := \max \Bigl\{ \,\sup_{y \in S} d(y, \tilde{S}), \sup_{\tilde{y} \in \tilde{S}}d(\tilde{y}, S)\, \Bigr\}
    ,$$
    \end{linenomath*}
    where $d(y,\tilde{S})$ is the distance from $y$ to its least-squares projection onto $\tilde{S}$. One may define a condition number that measures the error in this distance.
    To the best of our knowledge, this approach would be difficult to generalise to nonlinear problems, since it is computationally infeasible to compute the Pompeiu--Hausdorff distance for most sets.

    \item \label{item: affine Grassmannian} As before, the solutions are sets $S_x$, which are $(n-m)$-dimensional affine subspaces of $\mathbb{R}^n$. 
    The set of all such subspaces was called the \emph{affine Grassmannian} in \cite{limGrassmannianAffineSubspaces2021}. Since the affine Grassmannian is a Riemannian manifold, it has an induced distance (and hence a measure of foward error). The condition number with respect to this Riemannian distance can be studied using the techniques of \cite[Chapter 14]{Burgisser2013}.
    % \item The solutions are considered as sets $S_x$ as in \cref{eq: Sx linear system}. To measure the perturbation from $S_{x_0}$ to $S_{x}$, we project any solution $y_0 \in S_{x_0}$ orthogonally onto $S_{x}$. This projection is given by $y_0 \mapsto y_0 + \Delta y$, where $\Delta y = A^\dagger \Delta x$.
    % We define the condition number $\kappa$ as $\max_{\Delta x \ne 0} \frac{\norm{\Delta y}}{\norm{\Delta x}} =: \norm{A^\dagger}$, so that $\norm{\Delta y} \le \kappa \norm{\Delta x}$. 
    % Although $\kappa$ equals the condition number from item 1, its interpretation is very different.
\end{enumerate}

We propose a fourth, practical alternative, which does not require the space of solution sets to be a manifold (as in item \ref{item: affine Grassmannian}) and results in a condition number that can be computed using numerical linear algebra.
In our approach, we fix any solution $y_0$ corresponding to the noiseless input $x_0$. For a noisy input $x$, the condition number we propose satisfies the error bound 
\begin{equation}
    \label{eq: error bound LS cond informal}
\min_{y\colon (x,y) \in \mathcal P} d_{\mathcal Y}(y_0, y) \le \lscond[\mathcal P](x_0, y_0) \cdot d_{\mathcal X}(x_0, x) + o(d_{\mathcal X}(x_0, x))
\,\,\text{as}\,\, x \to x_0
,\end{equation}
where $\mathcal P = \{(x,y) \in \mathbb{R}^{m} \times \mathbb{R}^n \,\vert\, Ay = x\}$ and $\lscond[\mathcal P](x_0, y_0)$ is the proposed condition number at $(x_0, y_0)$. The left-hand side is the optimal forward error over all values of $y$ that solve $\mathcal P$ given $x$. The higher-order term $o(d_{\mathcal X}(x_0, x))$ can be neglected if $x$ is close to $x_0$.

\begin{figure}
    \centering
    \includegraphics[width=0.9\textwidth]{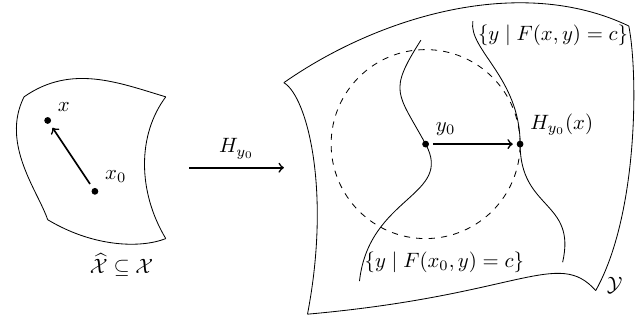}
    \caption{Simplified view of the solution sets of an FCRE. The point $y_0$ is a particular solution of $F(x_0,y) = c$ for the noiseless input $x_0$ and $x$ is a noisy input close to $x_0$. The point $H_{\y_0}(\x)$ is the projection of $\y_0$ onto the solution set $\{\y \mid F(\x,\y) = c\}$.
    If $F$ is linear, then the solution sets are affine spaces which are parallel to each other for all values of $x$.}
    \label{fig: LS solution map}
\end{figure}

We use this approach to derive an expression of the condition number of a wide class of problems we call \emph{feasible constant-rank equations (FCREs)} and write as $F(x,y) = c$ for some constant $c$. Deferring a precise definition to
\cref{sec: statement of main results}, FCREs can be defined informally as systems of (linear or non-linear) equations whose solution sets are smooth\footnote{Here and in the rest of the paper, \emph{smooth} means infinitely differentiable.} manifolds whose points depend smoothly on the input, generalising the linear problem above.
Many problems, such as matrix and tensor decompositions, can be modelled as FCREs, even though they are usually not thought of as a system of equations. A fortiori, if $G: \mathbb{R}^m \to \mathbb{R}^n$ is any polynomial map, the inverse problem $F(x,y) := G(y) - x = 0$ is an FCRE.
The error measure introduced in \cref{eq: error bound LS cond informal} is visualised for FCREs in \cref{fig: LS solution map}.
Our main result can be stated as follows. 

\begin{theorem}[informal version of \cref{thm: closest solution exists}]
    \label{thm: main theorem informal}
    Let $\mathcal{X}, \mathcal{Y}$, and $\mathcal Z$ be smooth manifolds where $\mathcal{X}$ and $\mathcal{Y}$ have a Riemannian metric. Let $c \in \mathcal Z$ be any constant and let $F: \mathcal X \times \mathcal Y \to \mathcal Z$ be a map so that the equation $F(x,y) = c$ is an FCRE. Finally, let $(x_0,y_0)$ be any pair that solves $F(x_0,y_0) = c$.
    Writing $\mathcal{P} := F^{-1}(c)$, the condition number in \cref{eq: error bound LS cond informal} is 
    \begin{linenomath*}
    $$
    \lscond[\mathcal{P}](x_0,y_0) = \norm{\left(\pdv{y}F(x_0,y_0)\right)^\dagger \pdv{x}F(x_0, y_0)}
    ,$$
    \end{linenomath*}
    where $\norm{\cdot}$ is the spectral norm.
\end{theorem}

The most rudimentary application of our results is the sensitivity of linear systems. Applying \cref{thm: main theorem informal} to the FCRE $Ay - x = 0$ gives $\norm{A^\dagger}$ as the (absolute) condition number. Note that this also the condition number of the problem $x \mapsto A^\dagger x$ described in item \ref{item: minimise solution norm} above \cite[eq. 12.2]{trefethenNumericalLinearAlgebra1997}.

We will show that relaxing a problem defined by an FCRE can never increase the condition number defined in this sense. The precise statement is as follows.
\begin{corollary}
    \label{corollary: relaxation decreases condition number}
    Given $R\colon \calX \times \widehat{\calY} \to \mathcal Z$ and $S\colon \calX \times \calY \to \mathcal Z$ where $\widehat{\calY} \subseteq \calY$ is a Riemannian submanifold of $\calY$, consider the FCREs $R(\x,\y) = \hat{c}$ and $S(\x,\y) = c$ and assume that the former is more constrained. If $R(\x_0,\y_0) = \hat{c}$ for some $(\x_0,\y_0)$, then
    \begin{linenomath*}
    \[
    \lscond[S^{-1}(c)](\x_0,\y_0) \le \lscond[R^{-1}(\hat{c})](\x_0,\y_0)
    .\]
    \end{linenomath*}
\end{corollary}

Hence, if a problem $\mathcal{P}$ can be relaxed to $\mathcal{P}'$, then the condition number of $\mathcal{P}'$ is a lower bound for the condition number of $\mathcal{P}$. This is essentially because relaxing a problem adds more possible solutions, so that the left-hand side of \cref{eq: error bound LS cond informal} may decrease but not increase.
This identity is useful for explaining the condition of a problem $\mathcal{P}$: if $\mathcal{P}$ can be relaxed to $\mathcal{P}'$, and $\mathcal{P}'$ is ill-conditioned, then so is $\mathcal{P}$. Conversely, if $\mathcal{P}$ is ill-conditioned but its relaxation $\mathcal{P}'$ is not, then the additional constraints that $\mathcal{P}$ adds to $\mathcal{P}'$ explain the condition of $\mathcal{P}$. 

The simplest instance of this is again the underdetermined linear system $Ay = x$ with $A \in \mathbb{R}^{m \times n}$, whose condition number is $\norm{A^\dagger}$ at any solution pair $(x_0, y_0)$. 
A more constrained system could introduce $n - m$ additional equations and  be written as $\overline{A}y = \overline{x}$ where $\overline{A} \in \mathbb{R}^{n \times n}$, the first $m$ rows of $\overline{A}$ are $A$, and the first $m$ components of $\overline{x}$ are $x$. The absolute condition number of this system is $\norm{\overline{A}^{-1}}$. By the singular value interlacing theorem \cite[Theorem 4.3.17]{hornMatrixAnalysis2012}, $\norm{\overline{A}^{-1}} \ge \norm{A^\dagger}$, which confirms \cref{corollary: relaxation decreases condition number}.

Suppose that a problem $\mathcal{P}$ has a solution pair $(x_0, y_0)$ and $\mathcal{P}'$ is a relaxation of $\mathcal{P}$.
We say that $\mathcal{P}$ is an \emph{optimal refinement} of $\mathcal{P}'$ at $(x_0, y_0)$ if relaxing $\mathcal{P}$ to $\mathcal{P}'$ does not decrease its condition number, i.e., $\lscond[\mathcal{P}](x_0, y_0) = \lscond[\mathcal{P}'](x_0, y_0)$. 
By the above example, if we have an underdetermined linear system $Ay = x$, then imposing the constraint that $\norm{y}$ be minimal gives an optimal refinement. Thus, our results justify the widespread use of this minimality constraint.

In another example, which we elaborate in \cref{sec: Tucker}, 
we present a family of matrices $\{X_{\varepsilon}\}_{\varepsilon \in \mathbb{R}^+}$ for which the computation of a singular value decomposition $X_{\varepsilon} = U\Sigma V^T$ can be arbitrarily ill-conditioned, but the less constrained \emph{orthogonal Tucker decomposition} (which is the same as a singular value decomposition, except it does not enforce $\Sigma$ to be diagonal) has a near-optimal condition number. Thus, the singular value decomposition is not an optimal refinement of the Tucker decomposition. That is, for a perturbation of $X_{\varepsilon}$, it is possible to find a Tucker decomposition close to the decomposition of $X_{\varepsilon}$, but only if its second factor is not diagonal.
We may say that the constraint that $\Sigma$ be diagonal is responsible for the difference in condition.

\subsection{Notation}
The $n \times n$ identity matrix is denoted by $\mathds{I}_n$. 
For a fixed $\x$, we write 
\begin{linenomath*}
\[
F_\x(\y) := F(\x,\y) \quad\text{and}\quad 
F^{-1}_\x(c) := \{y \mid F(\x,\y) = c\}. 
\]
\end{linenomath*}
We define the following manifolds: $\mathrm{St}(m,n) := \{U \in \mathbb{R}^{m \times n} \,\vert\, U^TU = \mathds{I}_n\}$ is the Stiefel manifold, $O(n) := \mathrm{St}(n,n)$ is the orthogonal group, and $\mathbb{R}^{m \times n}_k$ is the manifold of $m \times n$ matrices of rank $k$. The $k$th largest singular value of a matrix $A$ is $\sigma_k(A)$. The canonical basis vectors of $\mathbb{R}^{n}$ are $e_1,\dots,e_n$. The tangent space to a manifold $\mathcal M$ at $p$ is $T_p\mathcal{M}$. A generic vector in this space is denoted as $\dot{p}$.

\subsection{Summary of contributions and outline}

The main contribution of this work is an asymptotically sharp estimate for the optimal forward error of a general FCRE.
This error estimate is given by \cref{thm: closest solution exists} and \cref{prop: condition of inverse problems}, which are proved in \cref{sec: main proofs}. 
The advantage of our approach is that it requires little geometric information about the solution sets.
For certain underdetermined systems, though, the solution set can be seen as a unique point on a quotient manifold. We compare this point of view to our approach in \cref{sec: orth symmetries}. 

Another contribution is the computation of the condition number of two specific problems of independent interest: two-factor matrix decomposition and Tucker decomposition of tensors. They are studied in \cref{sec: 2fd} and \cref{sec: Tucker}, respectively. Numerical experiments for the accuracy of the error bound \cref{eq: 1st order error bound LS} in the case of the Tucker decomposition are presented in \cref{sec: experiments}.

We start with an overview of the theory of condition numbers of nonlinear equations in the next section.

\section{Classic theory of condition}
\label{sec: prelims}

Rice \cite{Rice1966} defined the condition number of a map $H: \mathcal{X} \to \mathcal{Y}$ between metric spaces at a point $x_0 \in \mathcal X$ as
\begin{linenomath*}
\begin{equation}
    \label{eq: def cond limsup}
\kappa[H](x_0) := \limsup_{x \to x_0} \frac{d_{\mathcal Y}(H(x_0), H(x))}{d_{\mathcal X}(x, x_0)}
\end{equation}
\end{linenomath*}
where $d_{\mathcal X}$ and $d_{\mathcal Y}$ are the distances in $\mathcal X$ and $\mathcal Y$, respectively.
Equivalently, $\kappa[H](x_0)$ is the smallest number $\kappa$ such that
\begin{linenomath*}
\begin{equation}
    \label{eq: 1st order err bound (general)}
d_{\calY}(H(\x_0), H(\x)) \le \kappa \cdot d_{\calX}(\x_0, \x) + o(d_{\calX}(\x_0, \x))
\quad\text{as}\quad 
\x \to \x_0
.\end{equation}
\end{linenomath*}
The latter property is called \emph{asymptotic sharpness} of the bound \cref{eq: 1st order err bound (general)}.
Since \cref{eq: def cond limsup} depends on $H$ (and consequently its domain) as well as the distances $d_{\mathcal X}$ and $d_{\mathcal Y}$, it is more precise to call $\kappa[H](x_0)$ \emph{the condition number of $H$ with respect to $d_{\mathcal X}$ and $d_{\mathcal Y}$}, but we refer to it simply as \emph{the condition number of $H$}. 
Many condition numbers in the literature are instances of \cref{eq: def cond limsup} for some choice of $\mathcal{X}, \mathcal{Y}$, and their distances. 
The following examples are common.
\begin{itemize}
    \item If $\mathcal X$ and $\mathcal Y$ are normed vector spaces, there is a natural distance $d_{\calX}(x, x') := \norm{x - x'}$ and likewise in $\mathcal Y$. In this case, \cref{eq: def cond limsup} is the \emph{absolute normwise condition number} $\kappa_{\text{abs}}[H]$. Alternatively, we may fix $x_0$ and define $d_{\cal X}(x, x') := \norm{x - x'}/\norm{x_0}$ and likewise in $\calY$. This corresponds to the \emph{relative normwise condition number} at $x_0$ \cite[Lecture 12]{trefethenNumericalLinearAlgebra1997}.
    \item If $\calX$ is not a linear space, then \cref{eq: def cond limsup} is often referred to as a \emph{structured} condition number \cite{hackbusch_interconnection_2017,Arslan2019}. For instance, if $\calX \subseteq \mathbb{R}^{n \times n}$ is a manifold of structured, orthogonal, or low-rank matrices, then \cref{eq: def cond limsup} only considers those matrices $x$ as possible perturbations. 
    \item If $\calX$ and $\calY$ are expressed in coordinates and one is interested in perturbations of only one coordinate, a \emph{componentwise} condition number may be used \cite{gohberg_mixed_1993}. This is less straightforward to define as a special case of \cref{eq: def cond limsup} and will not be considered in this paper.
\end{itemize}
We will consider general maps $H$ where $\calX$ and $\calY$ are Riemannian manifolds.
In this case, Rice's theorem \cite{Rice1966} says that $\kappa[H](x_0) = \norm{DH(x_0)}$ where $DH$ is the differential of $H$ and $\norm{\cdot}$ is the operator norm. In particular, for a map $H$ between Euclidean spaces, $\kappa_{\text{abs}}[H]$ is the spectral norm of the Jacobian matrix of $H$, as in \cite[Lecture 12]{trefethenNumericalLinearAlgebra1997}.

Rice's definition \cref{eq: def cond limsup} can be extended to numerical problems $\mathcal{P} \subseteq \calX \times \calY$ that cannot be described as a map $H: \calX \to \calY$, such as the solution of a polynomial equation $\sum_{i=0}^d a_i y^i = 0$, where $x = (a_0, \dots, a_d)$, $a_d \ne 0$, and $d \ge 2$. Such problems may have multiple isolated solutions. Demmel \cite{demmel_condition_1987} defined a condition number for univariate polynomial rootfinding, which was generalised by Shub and Smale \cite{shub_complexity_1993} to homogeneous polynomial systems and by B\"urgisser and Cucker \cite[Section 14.3]{Burgisser2013} to problems $\mathcal{P} \subseteq \calX \times \calY$ where $\calX$ and $\calY$ are Riemannian manifolds. 

This extended definition of the condition number goes as follows: let the problem $\mathcal{P}$ be defined by a system of equations $F(x,y) = c$ where $F: \calX \times \calY \to \mathcal{Z}$ is a smooth map (meaning infinitely differentiable) and $c \in \mathcal Z$ is a constant. If $(x_0, y_0)$ is any solution and $\pdv{y} F(x_0, y_0)$ is invertible, the implicit function theorem implies the existence of a unique smooth map, often called the \emph{solution map} $H: \widehat{\calX} \to \calY$ defined on a neighbourhood $\widehat{\calX} \subseteq \calX$ of $x_0$ such that $H(x_0) = y_0$ and $F(x, H(x)) = c$ for all $x \in \widehat{\calX}$. The condition number of $\mathcal{P}$ as discussed in \cite[Section 14.3]{Burgisser2013} is defined as $\kappa[F](x_0, y_0) := \kappa[H](x_0)$, where the right-hand side is given by \cref{eq: def cond limsup}. Working this out gives 
\begin{equation}
    \label{eq: formula cond BC}
\kappa[F](x_0,y_0) = \norm{\left(\pdv{y} F(x_0,y_0)\right)^{-1} \pdv{x} F(x_0,y_0)}
.\end{equation}

\begin{remark}
    \label{rmk: condition depends on solution}
    If the equation $F(x,y) = c$ has two different solutions $y_0$ and $y_0'$ for the same input $x_0$, then their solution maps are different and may generally have a different condition number. Hence, the condition number may depend on the solution as well as the input. 
\end{remark}

\begin{example}[Eigenvalue problems]
    Let $\mathcal{X} := \mathbb{C}^{n \times n}$ and $\mathcal{Y} = \mathbb{C}$ and define $F(X, \lambda) := \det(X - \lambda\mathds{I})$. The eigenvalues of $X$ are precisely the solutions of $F(X,\lambda) = 0$.
    Endow $\mathbb{C}^{n}$ with the Hermitian inner product $\left\langle\cdot,\cdot\right\rangle$ and $\mathbb{C}^{n \times n}$ with the spectral norm.
    Let $X_0$ be a matrix with a simple eigenvalue $\lambda_0$ and left and right eigenvectors $v$ and $w$, respectively. Then $\kappa[F](X_0,\lambda_0) = \frac{\norm{v}\norm{w}}{\left\vert \left\langle v, w\right\rangle\right\vert}$.
    This was shown in \cite[IV--Theorem 2.3]{Stewart1990} using Gershorin's theorem and in \cite[Proposition 14.15]{Burgisser2013} using the implicit function theorem. For a generic $X$, all its eigenvalues have a different condition number, as per \cref{rmk: condition depends on solution}.
\end{example}

\subsection{Related work}

The condition number of systems with unique solutions is well understood \cite[Chapter 14]{Burgisser2013}. We know of two works studying a condition number in the sense of \cref{eq: error bound LS cond informal}.
First, Dedieu \cite{dedieu_approximate_1995} introduced the \emph{inverse condition number} of a numerical problem $G: \calY \to \calX$, where $\calY$ and $\calX$ are Euclidean spaces. The expression for this condition number is equivalent to \cref{prop: condition of inverse problems}, but its interpretation is different. Dedieu interpreted $\calY$ as the input space, $\calX$ as the output space, and was interested in measuring backward errors. Conversely, we study the equation $G(\y) = \x$ as a problem taking $\calX \to \calY$ and consider the forward error as in \cref{eq: 1st order error bound LS}. 
Another occurrence of a latent condition number is due to Vannieuwenhoven, who studied the sensitivity of the tensor rank decomposition in its factor matrix representation \cite{Vannieuwenhoven2017}.
\Cref{thm: main theorem informal} is a generalisation of both of these results. 
Another approach for defining a condition number of certain underdetermined systems is based on quotient manifolds. We explain this in detail in \cref{sec: orth symmetries}.

For several problems in numerical analysis, there is a connection between first-order sensitivity as in \cref{eq: 1st order err bound (general)}, the distance to the nearest ill-posed problem \cite{demmel_condition_1987}, and the convergence of iterative algorithms \cite{shub_complexity_1993}. 
Constants appearing in estimates of any of these measures are often called condition numbers, even if the asymptotic sharpness of the estimate is not demonstrated.
Specifically, D\'egot \cite{degot_condition_2000} introduced a condition-like number for underdetermined homogeneous polynomial systems that measures distance to ill-posedness. The same number provides an error estimate of the solution, but little was said about the asymptotic sharpness of this estimate. Dedieu and Kim \cite{dedieu_newtons_2002} analysed a generalised Newton method for solving the equation $G(x) = 0$, where $\rank DG(x)$ is constant. The rate of convergence can be estimated in terms of $\norm{DG(x)^\dagger}$, i.e., the expression appearing in \cref{prop: condition of inverse problems}.
In the context of linear least-squares problems of the form $Ax = b$, the similar expression $\norm{A}\norm{A^\dagger}$ is sometimes referred to as the condition number of $A$, even if this number is not the condition number of the problem as defined by \cref{eq: def cond limsup} (see e.g. \cite[Corollary III.3.10]{Stewart1990} and the discussion thereafter).

Another conceptually similar condition number is that of Riemannian approximation \cite{Breiding2019}. In that context, the problem is to project a variable point $x \in \mathbb{R}^n$ onto a fixed manifold $\mathcal M \subseteq \mathbb{R}^n$. The latent condition number, by contrast, measures how a fixed point $\y_0$ is projected onto a variable solution set $F_{\x}^{-1}(c)$.

\section{Proposed theory of condition}
\label{sec: statement of main results}
The condition number in the sense of Rice is defined for the evaluation of a map $H: \calX \to \calY$ and for the solution of systems of equations $F(x,y) = c$ where $\pdv{y}F(x,y)$ is invertible. In this section, we loosen this constraint on $\pdv{y}F(x,y)$ and propose a corresponding condition number.
\begin{definition}
    \label{def: FCRE}

    Let $\calX, \calY, \mathcal{Z}$ be smooth manifolds and let $c \in \mathcal Z$ be a constant. Let $F: \calX \times \calY \rightarrow \mathcal Z$ be a smooth map. We call the equation $F(\x,\y) = c$ a \emph{feasible constant-rank equation (FCRE)} if the following holds:
    \begin{enumerate}
        \item for all $\x \in \calX$, there exists a point $\y \in \calY$ such that $F(\x,\y) = c$,
        \item there exists a number $r \in \mathbb{N}$ so that $\rank \pdv{\y} F(\x, \y) = r$ and $\rank DF(\x,\y) = r$ for all $\x, \y \in \calX \times \calY$.
    \end{enumerate}
\end{definition}

Note that if $r = \dim \calY = \dim \mathcal Z$, then $\pdv{y}F(x,y)$ is invertible, which is the usual condition under which the condition number in the sense of \cite[Section 14.3]{Burgisser2013} is defined.
Note as well that if a map $F$ only satisfies this definition locally, the restriction of $F$ to a subset of its domain defines an FCRE.

The idea behind the second item in \cref{def: FCRE} is as follows.
If $\rank DF$ is constant, then $F$ is a \emph{map of constant rank}, which is a fundamental concept in differential geometry \cite[Chapter 4]{Lee2013}. One such example are polynomial maps: if $F: \mathbb{R}^m \to \mathbb{R}^n$ is a polynomial, then the locus of points $p \in \mathbb{R}^m$ such that $\rank DF(p)$ is \emph{not} maximal is a subvariety of $\mathbb{R}^m$ of dimension less than $m$. Thus, polynomials have constant rank almost everywhere \cite[Proposition A.35]{Burgisser2013}.

For any map $F$ of constant rank, the problem $\mathcal{P} := F^{-1}(c)$ is a smooth manifold of dimension $\operatorname{null}DF$ where $\operatorname{null}$ is the nullity \cite[Theorem 4.12]{Lee2013}. It then follows that $\rank \pdv{F}{y} = \rank DF$ if and only if $\dim \mathcal{P} = \dim \mathcal{X} + \operatorname{null} \pdv{F}{y}$. This should be compared to identifiable problems, which have only $\dim \mathcal X$ degrees of freedom (since every $x \in \mathcal{X}$ would determine a unique $(x,y) \in \mathcal{P}$). By the foregoing, FCREs are defined exactly by the maps of constant rank that offer $\operatorname{null} \pdv{F}{y}$ additional degrees of freedom.

Our main theorem underlies the definition of the condition number of an FCRE. It is proved in \cref{sec: main proofs}.
\begin{theorem}
    \label{thm: closest solution exists}

    Let $F(\x,\y) = c$ be an FCRE as in \cref{def: FCRE} and let $(\x_0, \y_0)$ be any pair such that $F(\x_0,\y_0) = c$. If $\mathcal X$ and $\mathcal Y$ are Riemannian manifolds, then there exist a neighbourhood $\widehat{\calX} \subseteq \calX$ of $\x_0$ and a smooth map, called the \emph{canonical solution map}
    \begin{linenomath*}
    \begin{align*}
    H_{\y_0}: \widehat{\calX} &\to \calY \\
    \x &\mapsto \arg\min_{\substack{
        \y \in \calY\\
        F(\x,\y) = c
    }} d_{\calY}(\y_0, \y), 
    \end{align*}
    \end{linenomath*}
    where $d_{\calY}$ is the geodesic distance in $\calY$. Its differential at $x_0$ is
    \begin{linenomath*}
    $$
    DH_{y_0}(x_0) = 
    \left(\pdv{y}F(x_0,y_0)\right)^\dagger \pdv{x}F(x_0,y_0)
    .$$
    \end{linenomath*}
\end{theorem}

That is, out of all possible solution maps, $H_{\y_0}$ locally minimises the distance to $\y_0$. \Cref{fig: LS solution map} shows a visualisation of $H_{\y_0}$. The preceding theorem allows us to define our primary object of interest.

\begin{definition}
    \label{def: cond FCRE}
    In the context of \cref{thm: closest solution exists}, the \emph{latent condition number} of $F$ at $(\x_0, \y_0)$ is
    \begin{linenomath*}
        \[
    \lscond[F^{-1}(c)](\x_0, \y_0) := \kappa[H_{\y_0}](\x_0)
    = \norm{
        \left(\pdv{y}F(x_0,y_0)\right)^\dagger \pdv{x}F(x_0,y_0)}
    \]
    \end{linenomath*}
    where $\norm{\cdot}$ is the operator norm.
\end{definition}

Note that this is exactly the generalisation of \cref{eq: formula cond BC} that would be obtained if $\pdv{F}{y}$ in \cref{eq: formula cond BC} were naively replaced by a pseudoinverse. The value of \cref{thm: closest solution exists} is that it shows that this generalised formula has a precise interpretation: $\lscond$ expresses whether the equation $F(\x,\y) = c$ has a solution close to $\y_0$ if $\x$ is a slight perturbation of $\x_0$. If $d_{\calX}$ and $d_{\calY}$ are the geodesic distances in $\calX$ and $\calY$, respectively, we have the following asymptotically sharp error bound:
\begin{linenomath*}
\begin{equation}
    \label{eq: 1st order error bound LS}
\min_{\substack{\y \in \calY,\\ F(\x,\y) = c}} d_{\calY}(\y_0, \y) \le \lscond[F^{-1}(c)](\x_0, \y_0) \cdot d_{\calX}(\x_0, \x) + o(d_{\calX}(\x_0, \x)) 
\,\,\text{as}\,\, 
\x \to \x_0
.\end{equation}
\end{linenomath*}
Since this is a bound on the asymptotic behaviour as $\x \to \x_0$, the same bound holds for any distance $d$ such that $d(\x_0, \x) = d_{\calX}(\x_0, \x)(1 + o(1))$ as $\x \to \x_0$ and likewise for $d_{\calY}$. For instance, if $\mathcal{X}$ is an embedded Riemannian submanifold of a Euclidean space $\mathcal{E}$, we may take $d$ to be the Euclidean distance in $\mathcal{E}$. Then \cref{eq: 1st order error bound LS} gives a bound in the (more practical) Euclidean distance $d$.

\begin{remark}
    The classic condition number of equations on manifolds, as defined in \cite[Section 14.3]{Burgisser2013}, requires a \emph{unique} solution map at the given solution pair $(x_0, y_0)$.
    If this map does not exist or is not unique, the condition number is either undefined or infinite by definition \cite{Breiding2018a}.
    If the classic condition number of an FCRE is finite, the unique solution map is the map from \cref{thm: closest solution exists} and the condition number is the latent condition number.
\end{remark}

The condition number of a relaxation of a FCRE can be used as a lower bound for the condition number of the original FCRE, as stated in \cref{corollary: relaxation decreases condition number}.
In other words, if a (relaxed or underdetermined) problem has a high latent condition number, then all ways to make it well-posed by adding constraints will be ill-conditioned. This is the intuition behind the name \emph{latent condition number}.
\Cref{corollary: relaxation decreases condition number} is a straightforward consequence of the definition, but we prove it here for completeness.

\begin{proof}[Proof of \cref{corollary: relaxation decreases condition number}.]
    Let $H_R$ and $H_S$ be the solution maps arising from the application of \cref{thm: closest solution exists} to $R$ and $S$, respectively. For all $x \in \calX$ sufficiently close to $x_0$, we have 
    \begin{align*}
    d_{\calY}(H_S(x_0), H_S(x)) &= \min_{
        \substack{
            y \in \mathcal Y\\
            S(x,y) = c}
        } d_{\calY}(y_0, y)
    .\end{align*}
    An upper bound on this can be obtained by restricting the domain of the minimum and applying the identity that $d_{\calY}(y, y') \le d_{\widehat{\calY}}(y,y')$ for all $y,y' \in \calY$.
    Thus, 
    \begin{linenomath*}
    $$
    d_{\calY}(H_S(x_0), H_S(x)) \le \min_{
        \substack{
            y \in \widehat{\mathcal Y}\\
            R(x,y) = \hat{c}}
        } d_{\widehat{\calY}}(y_0, y)
    =
    d_{\widehat{\calY}}(H_R(x_0), H_R(x))
    $$
    \end{linenomath*}
    so that the result follows from \cref{eq: def cond limsup}.
\end{proof}

An important class of systems of equations are equations of the form $G(\y) - \x = 0$ for some smooth map $G: \calY \to \mathbb{R}^m$. 
In this case, \cref{thm: closest solution exists} specialises to the following Riemannian generalisation of \cite[Theorem C]{dedieu_approximate_1995}.

\begin{corollary}
    \label{prop: condition of inverse problems}
    Let $\calY$ be a Riemannian manifold and let $G: \calY \to \mathbb{R}^{m}$ be a smooth map such that $\rank DG(\y)$ is constant.
    Pick any point $(\x_0,\y_0)$ on the graph of $G$.
    Then $\y_0$ has a neighbourhood $\widehat{\calY}$ such that $\calX := G(\widehat{\calY})$ is an embedded submanifold of $\mathbb{R}^{m}$.
    Define
    \begin{linenomath*}
    \begin{align*}
        F: \calX \times \widehat{\calY} &\to \mathbb{R}^{m} \\
        (\x,\y) &\mapsto G(\y) - \x
    .\end{align*}
    \end{linenomath*}
    Then $F(\x,\y) = 0$ is an FCRE and $\lscond[F^{-1}(0)](\x_0,\y_0) = \norm{DG(\y_0)^\dagger}$.
\end{corollary}
For a map $G$ satisfying the assumptions in \cref{prop: condition of inverse problems}, we call the equation $G(\y) - \x = 0$ a \emph{constant-rank inverse problem} and we write $\lscond^{inv}[G](\y_0) := \lscond[F^{-1}(c)](\x_0,\y_0)$.
The dependence on $x_0$ is not written explicitly since $x_0$ is determined by $y_0$.
Note that $\norm{DG(\y_0) ^\dagger}$ is the reciprocal of the smallest nonzero singular value of $DG(\y_0) $.

\subsection{Proof of \cref{thm: closest solution exists}}
\label{sec: main proofs}

The proof of \cref{thm: closest solution exists} is an application of standard concepts from differential geometry and numerical analysis. We will use the following lemma to parameterise the tangent space to the solution sets.

\begin{lemma}
    \label{lemma: vector fields kernel}
    Let $\calX, \calY, \mathcal Z$ be smooth manifolds of dimensions $m,n$ and $k$, respectively, and let $F: \calX \times \calY \to \mathcal Z$ be a smooth map. Suppose that $\rank \frac{\partial}{\partial y} F(\x,\y) = r$ is constant. In a neighbourhood of any point $(\x_0,\y_0) \in \calX \times \calY$, there exists a linearly independent tuple of smooth vector fields $((0,E_1(\x,\y)),\dots,(0,E_{n - r}(\x,\y)))$ over $\calX \times \calY$ whose span is $\{0\} \times \ker \pdv{y}F(\x,\y)$.
\end{lemma}
\begin{proof}
    Consider the map 
    \begin{linenomath*}
    \begin{align*}
        \tilde{F}:\, \calX \times \calY &\to \calX \times \mathcal Z\\
        (\x,\y) &\mapsto (\x, F(\x,\y)).
    \end{align*}
    \end{linenomath*}
    Then, $(\dot{\x},\dot{\y})\in \ker D\tilde{F}(\x,\y)$ if and only if $\dot{\x}=0$ and $(\dot\x,\dot\y) \in \ker D F(\x,\y)$. Since $D F(x,y)[\dot\x,\dot\y] = \frac{\partial}{\partial\x} F(x,y)[\dot\x] + \frac{\partial}{\partial\y} F(x,y)[\dot\y]$, $\dot\x=0$, and $\frac{\partial}{\partial\x}F(x,y) : T_x \calX \to T_{F(x,y)}\mathcal{Z}$ is a linear map, we have 
    \begin{equation}
        \label{eq: kernel extended F}
    \ker D\tilde{F}(\x,\y) = \{0\} \times \ker \frac{\partial}{\partial y} F(\x,\y) 
    \end{equation} and $\rank D\tilde{F}(\x,\y) = m + r$ for all $\x,\y$.

    By the constant rank theorem \cite[Theorem 4.12]{Lee2013}, there exist charts for the domain and codomain of $\tilde{F}$ in which $\tilde{F}$ is represented as 
    \begin{linenomath*}
    \[
    (u^1,\dots,u^{m + n}) \mapsto (u^1,\dots,u^{m+r},0,\dots,0)
    .\] 
    \end{linenomath*}
    In these coordinates, the basis $\mathcal{B} := \left\{ \pdv{u^i} \right\}_{i=m + r + 1}^{m + n}$ spans $\ker D\tilde{F}(\x,\y)$. By \cref{eq: kernel extended F}, we may write these $\pdv{u^i}$ as smooth vector fields $(0, E_{i - m - r}(\x,\y))$, where $E_{i-m-r}(\x,\y) \in \ker \pdv{y}F(\x,\y)$. % Thus, the desired vector fields are the projections of $\mathcal{B}$ onto $T_\y \calY$.
\end{proof}

Now we can prove the existence of the canonical solution map.
\begin{proof}[Proof of \cref{thm: closest solution exists}]
    Let $n = \dim \calY$ and let $g_{\calY}$ be the Riemannian metric on $\calY$.
    Let $\{E_i(\x,\y)\}_{i=1}^{n - r}$ be the vector fields from \cref{lemma: vector fields kernel}. 

    By the constant rank theorem \cite[Theorem 4.12]{Lee2013}, there exists a neighbourhood $\mathcal U \subseteq \calX \times \calY$ of $(\x_0, \y_0)$ and a chart $\phi_{\mathcal Z}: \mathcal Z \to \mathbb{R}^{\dim \mathcal Z}$ such that $\phi_{\mathcal Z}(c) = 0$ and $\phi_{\mathcal Z}(F(\x,\y)) = (\widehat{F}(\x,\y), 0)$ for some smooth map $\widehat{F}: \mathcal U \to \mathbb{R}^r$. Thus, in this neighbourhood, the equation $F(\x,\y) = c$ is equivalent to $\widehat{F}(\x,\y) = 0$.

    Let $\log_\y: \calY \to T_{\y}\calY$ be the inverse of the exponential map in $\calY$. Informally, $\log_y y_0$ is the vector in $T_y \calY$ that ``points towards'' $\y_0$. Define $\phi_i(\x,\y) := g_{\calY}(E_i(\x,\y), \log_\y \y_0)$ and consider the system of $n$ equations
    \begin{equation}
        \label{eq: critical point system}
    \Phi(\x,\y)
        :=
    (\widehat{F}(\x,\y),\,\phi_1(\x,\y),\dots, \phi_{n-r}(\x,\y))
    =
    (0,0,\dots,0)
    .\end{equation}
    The last $n-r$ equations specify that $\log_\y \y_0$ is orthogonal to $\ker \pdv{\y} F(\x,\y)$ and thus normal to $F_\x^{-1}(c)$.
    We will show, using the implicit function theorem, that \cref{eq: critical point system} has a locally unique solution.
    
    Let $g_{\calX \times \calY}$ be the product metric in $\calX \times \calY$. Then
    \begin{linenomath*}
    \[
        \phi_i(\x,\y) = g_{\calX \times \calY}\left((0, E_i(\x,\y)), (0, \log_\y \y_0)\right)
    .\]
    \end{linenomath*}
    Let $(\xi, \eta) \in T_{(\x_0, \y_0)} (\calX \times \calY)$ be any tangent vector and let $\nabla$ be the Levi-Civita connection for $g_{\calX \times \calY}$. We calculate $D\phi_i(\x_0, \y_0)$ using the product rule:
    \begin{linenomath*}
    \begin{multline}
        \label{eq: Dphi_i}
        D\phi_i(\x_0,\y_0)[\xi, \eta] = g_{\calX \times \calY}\left(\nabla_{(\xi,\eta)} (0, E_i(\x,\y)), (0, \log_{\y_0} \y_0)\right)  \\
            + g_{\calX \times \calY}\left((0, E_i(\x_0,\y_0)), \nabla_{(\xi,\eta)} (0, \log_y y_0)\right)
    .\end{multline}
    \end{linenomath*}
    The first term vanishes because $\log_{y_0} y_0 = 0$. 
    The second term simplifies to $g_{\calY}(E_i(\x_0, \y_0), \nabla_\eta \log_\y \y_0)$. Using normal coordinates centred at $\y_0$, the vector field $\log_\y \y_0$ can be written as $\log_\y \y_0 = - \sum_{i=1}^{n} y^i \pdv{y^i}$, so that $\nabla_\eta \log_\y \y_0 = -\eta$ \cite[Proposition 5.24]{Lee2013}. Hence, \cref{eq: Dphi_i} is equal to $-g_{\calY}(E_i(\x_0, \y_0), \eta)$.

    To apply the implicit function theorem to \cref{eq: critical point system}, we verify that $\pdv{\y} \Phi(\x_0,\y_0)$ is invertible. It suffices to show that the kernel of $\pdv{\y} \Phi(\x_0,\y_0)$ is trivial. If a vector $\dot{\y} \in T_{\y_0}\calY$ is such that $\pdv{\y}\Phi(\x_0,\y_0)[\dot{\y}] = 0$, then $\dot{\y} \in \ker \pdv{\y}\widehat{F}(\x_0,\y_0) = \ker \pdv{\y} F(\x_0,\y_0)$. Furthermore, if $\pdv{\y}\phi_i(\x_0, \y_0)[\dot\y] = 0$ for all $i$, then $\dot\y$ is orthogonal to $E_i(\x_0, \y_0)$ for all $i$. By the definition of $E_i$, it follows that $\dot\y \perp \ker \pdv{\y} F(\x_0, \y_0)$ and thus $\dot\y = 0$. Therefore, $\ker \pdv{\y} \Phi(\x_0,\y_0) = \{0\}$.
    By the implicit function theorem \cite[Theorem C.40]{Lee2013}, there exists a neighbourhood $\widehat{\calX} \times \widehat{\calY}$ of $(\x_0, \y_0)$ and a smooth function $H_{\y_0}$ such that $\Phi(\x, \y) = 0$ for $(\x, \y) \in \widehat{\calX} \times \widehat{\calY}$ if and only if $\y = H_{\y_0}(\x)$.

    Next, we show that $H_{\y_0}(\x)$ is the map from the theorem statement. Consider a variable point $\x \in \widehat{\calX}$. By continuity of $H_{\y_0}$, if $\x$ is sufficiently close to $\x_0$, then $H_{y_0}(\x)$ lies in the interior of some compact geodesic ball $\overline{B} \subseteq \widehat{\calY}$ of radius $\rho$ around $\y_0$. Since the level set $F_\x^{-1}(c)$ is properly embedded \cite[Theorem 4.12]{Lee2013}, the minimum of $d_{\y_0}(\y) := d_{\mathcal Y}(\y_0, \y)$ over all $\y \in F_\x^{-1}(c) \cap \overline{B}$ is attained. The interior of $F_\x^{-1}(c) \cap \overline{B}$ contains at least $H_{\y_0}(\x)$ and, since $d(\y_0, H_{\y_0}(\x)) < \rho$, it follows that $d_{\y_0}$ attains a minimum in the interior of $F_\x^{-1}(c) \cap \overline{B}$. Thus, at the minimiser $\y_\star$, we must have 
    \begin{linenomath*}
    \[
    \mathrm{grad} \, d_{\y_0}^2(\y) \perp T_{\y_\star} F_{\x}^{-1}(c) = \ker \pdv{\y}F(\x,\y_\star),
    \]
    \end{linenomath*}
    where $\mathrm{grad}\, d_{\y_0}^2(\y) = -2 \log_y \y_0$. As we established above, $H_{\y_0}(\x)$ is the unique point that solves \cref{eq: critical point system}. In other words, it is the only $\y \in F_\x^{-1}(c) \cap \widehat{\calY}$ such that $\log_\y \y_0 \perp \ker \pdv{\y} F(\x, \y)$. Thus, $H_{\y_0}(\x) = \y_\star$, as required.

We obtain an expression for $DH_{\y_0}(\x_0)$ as follows. Since $\Phi(\x, H_{\y_0}(\x))$ is constant for all $\x$, it follows by implicit differentiation that
\[
\pdv{\x} \Phi(\x_0, \y_0) +
\pdv{\y} \Phi(\x_0, \y_0) DH_{\y_0}(\x_0) = 0
.\]
By substituting the partial derivatives of $\Phi$ obtained in the proof, we get
\begin{linenomath*} 
\[
\begin{cases}
    \pdv{\x}F(\x_0,\y_0) + \pdv{\y} F(\x_0,\y_0) DH_{\y_0}(\x_0) &= 0, \\
    E_i^*(\x_0,\y_0) DH_{\y_0}(\x_0) &= 0 \quad\text{for all}\quad i=1,\dots,n-r,
\end{cases}
\]
\end{linenomath*}
where $\cdot^*$ is the dual (or adjoint). In other words, $DH_{\y_0}(\x_0)$ is the unique matrix that solves $\pdv{\y} F(\x_0,\y_0) DH_{\y_0}(\x_0) = - \pdv{\x}F(\x_0,\y_0)$ and has a column space orthogonal to $\ker \pdv{\y} F(\x_0,\y_0)$.
Hence, $DH_{\y_0}(\x_0) = -\left(\pdv{\y} F(\x_0,\y_0)\right)^\dagger \pdv{\x}F(\x_0,\y_0)$.
\end{proof}

\section{Problems invariant under orthogonal symmetries}
\label{sec: orth symmetries}

A notable advantage of the latent condition number of an FCRE $F(\x,\y) = c$ is that it only requires information about the local behaviour of $F$ around a particular solution $(\x_0, \y_0)$. In particular, it does not require an explicit parametrisation of all solutions in terms of $(\x_0, \y_0)$. By contrast, for some underdetermined systems studied in the literature, the derivation of their condition number relies on the solutions being unique up to a known equivalence relation, as in \cite[Section 14.3]{Burgisser2013} and \cite{Vannieuwenhoven2017}.

When it \emph{is} known that the system is invariant under certain symmetries, however, more can be said about the condition numsber.
For instance, since the condition number generally depends on the solution $\y$ and the parameter $\x$, it is natural to ask when it depends on the parameter alone. That is, when do two distinct $\y_1, \y_2$ that solve $F(\x,\y) = c$ for the same $\x$ satisfy $\lscond[F^{-1}(c)](\x, \y_1) = \lscond[F^{-1}(c)](\x, \y_2)$? An obvious sufficient condition for this is that both $F$ and its solutions are invariant under some family of isometries. This is captured by the following statement.

\begin{proposition}
    \label{prop: condition invariant under isometry}
    Let $F(\x,\y) = c$ be an FCRE with $F: \calX \times \calY \to \mathcal Z$. Let $\psi: \calY \to \calY$ be an isometry such that $F \circ (\mathrm{Id} \times \psi) = F$. For any $\x,\y \in \calX \times \calY$, we have 
    \begin{linenomath*}
    \[
    \lscond[F^{-1}(c)](\x,\y) = \lscond[F^{-1}(c)](\x, \psi(\y))
    .\]
    \end{linenomath*}
\end{proposition}
\begin{proof}
    Compute
    \begin{linenomath*}
    \[
    DF(\x,\y) = D(F \circ (\mathrm{Id} \times \psi))(\x,\y) = DF(\x, \psi(\y)) (\mathrm{Id} \times D\psi(\y))
    \]
    \end{linenomath*}
    so that $\pdv{\y} F(\x,\y) = \pdv{\y} F(\x, \psi(\y)) D\psi(\y)$ and $\pdv{\x} F(\x,\y) = \pdv{\x} F(\x, \psi(\y))$.
    Since $D\psi(\y)$ is an orthogonal matrix, applying \cref{thm: closest solution exists} gives the desired result.
\end{proof}
This proposition is useful when the solutions are determined up to certain isometries.
That is, suppose that $\x \in \calX$ is any point and $\{\psi_i\}_{i \in I}$ is a family of isometries such that $F = F \circ (\mathrm{Id} \times \psi_i)$ for all $i$. If, for every $\y_1, \y_2$ where $F(\x, \y_1) = F(\x, \y_2) = c$, there exists an $i \in I$ such that $\y_1 = \psi_i(\y_2)$, then the above implies that all solutions of $F(\x,\y) = c$ have the same condition number.

For several problems in numerical linear algebra, the solutions are unique up to multiplication by an orthogonal matrix, i.e., a linear isometry. Thus, their condition number depends only on the input by \cref{prop: condition invariant under isometry}. Some examples include:
\begin{enumerate}
    \item \textit{Positive-semidefinite matrix factorisation}: $\calX = (\mathbb{S}^{n \times n}_k)^+$ is the set of symmetric positive semidefinite matrices of rank $k$ and $\calY = \mathbb{R}^{n \times k}_k$. A symmetric factorisation of $\X \in \calX$ is a solution of $F(\X,\Y) = 0$, where $F(\X,\Y) := \X - \Y\Y^T$. 
    \item \label{itm: kernel computation} \textit{Computation of an orthonormal basis of the kernel}: $\calX = \mathbb{R}^{m \times n}_k$ and $\calY = \mathrm{St}(n,n-k)$ and $F(\X,\Y) = 0$, where $F(\X,\Y) := \X\Y$. 
    \item \textit{Computation of an orthonormal basis of the column space}: if $\calX = \mathbb{R}^{m \times n}_k$  $\calY = \mathrm{St}(m, k)$, then $\Y \in \calY$ is a basis of $\operatorname{span} \X$ for some $\X \in \calX$ if and only if $F(\X,\Y) := (\mathds{I}_m - \X\X^\dagger)Y = 0$. 
    \item \textit{Orthogonal Tucker decomposition}: see \cref{sec: Tucker}.
\end{enumerate}
In the first three examples, the solution $Y$ is unique up to the isometries $\psi: \Y \mapsto \Y Q$, where $Q$ is any orthogonal matrix in $O(k)$.

\subsection{Comparison to the quotient-based approach}
\label{sec: LS vs quotient}

If the solutions to a problem are invariant under a known symmetry group, they can be considered as uniquely defined points in a quotient space as opposed to a set of many solutions. For example, consider the problem of computing the eigenvector corresponding to a given simple eigenvalue of a matrix if $A \in \mathbb{C}^{n \times n}$. Depending on the precise formulation of the problem, the solution can either be considered a set of points in $\mathbb{C}^n$ or as a unique point in projective space.

For some underdetermined problems, a notion of condition has been worked out by quotienting out symmetry group of the solution set \cite[Section 14.3]{Burgisser2013}. The fundamentals of this technique are recapped below.
In the remainder of this section, we investigate whether the condition number arising from this method agrees with the latent condition number. 

Suppose that $F(\x,\y) = c$ is an FCRE and that there exists an equivalence relation $\sim$ so that $F(\x, \y) = F(\x, \y')$ for all $\x$ if and only if $\y \sim \y'$. If $\pi: \calY \to \calY / {\sim},\, \y \mapsto [\y]$ is the projection of a point onto its equivalence class, there exists a unique map $\widetilde{F}$ such that the following diagram commutes.
\begin{equation}
    \label{eq: cd quotient map}
    \includegraphics{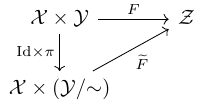}
\end{equation}
Under certain conditions, the projection map $\pi$ and the metric in $\calY$ induce a Riemannian structure on $\calY / {\sim}$.
That is, the differential $D\pi$ is a formal linear map such that every smooth function $G \circ \pi$ where $G$ is of the form $\calY/{\sim} \to \mathcal Z$ obeys the chain rule (as in \cite[Theorem 4.29]{Lee2013}) and the restriction of $D\pi$ to the orthogonal complement of its kernel is a linear isometry. In this case, $\pi$ is called a \emph{Riemannian submersion}. For example, the orbits of certain groups acting isometrically on $\calX$ form a Riemannian manifold such that the quotient projection is a Riemannian submersion \cite[Theorem 2.28]{Lee2018}.

Riemannian submersions give an alternative perspective on the system $F(\x,\y) = c$: it can be formulated equivalently as $\widetilde{F}(\x, [\y]) = c$ where the goal is to solve for a representative of $[\y]$. For this equation, the condition number at a point $(\x_0, [\y_0])$ is given by \cite[Section 14.3]{Burgisser2013}:
\begin{equation}
    \label{eq: cond quotient manifold}
\kappa[\widetilde{F}](\x_0, [\y_0])
=
\norm{\left(\pdv{[\y]} \widetilde{F}(\x_0, [\y_0])\right)^{-1} \pdv{\x}\widetilde{F}(\x_0, [\y_0])}
.\end{equation}
This can be pulled back to a more concrete expression over the original domain $\calX \times \calY$, so that the derivative over the quotient space is not explicitly needed.
Because of the way the metric in $\calY / {\sim}$ is defined, the above turns out to be equal to latent condition number by the following proposition.

\begin{proposition}
    \label{prop: LS is equivalent to quotient}
    Let $F(\x,\y) = c$ be an FCRE, where $F: \calX \times \calY \to \mathcal Z$ is smooth. Let $\pi: \calY \to \calY/{\sim},\, \y \mapsto [\y]$ be a Riemannian submersion such that \cref{eq: cd quotient map} commutes. 
	Assume that $\ker \pdv{\y} F(\x,\y) = \ker D\pi(\y)$
	and that $\pdv{[\y]} \widetilde{F}(\x, [\y])$ is invertible.
	Then, at every $(\x_0,\y_0) \in \calX \times \calY$, we have 
	\begin{linenomath*}
    \[
	\lscond[F^{-1}(c)](\x_0, \y_0) = \kappa[\widetilde{F}](\x_0,[\y_0])
	\]
    \end{linenomath*}
    where the right-hand side is given by \cref{eq: cond quotient manifold}.
\end{proposition}%
\begin{proof}
    Define $H_0 := (\ker D\pi(\y_0))^\perp = \left(\ker \pdv{\y} F(\x_0,\y_0)\right)^\perp$.
	By the definition of a Riemannian submersion, $H_0$ is isometric to $T_{[\y_0]} (\calY / {\sim})$.
	Thus, we may write 
	\begin{linenomath*}
    \[
	\pdv{[\y]} \widetilde{F}(\x_0, [\y_0]): 
	H_0
	\to 
	T_{F(\x_0,\y_0)} \mathcal Z,
	\]
    \end{linenomath*}
	so that $\pdv{[\y]} \widetilde{F}(\x, [\y]) = \pdv{\y} F(\x_0,\y_0)\bigr\vert_{H_0}$.
	If $A$ is a surjective linear map, then $A^\dagger$ is the inverse of the restriction of $A$ to $(\ker A)^\perp$. Thus, 
	\begin{linenomath*}
        \[
	\paren{ \pdv{\y} F(\x_0,\y_0) }^\dagger = \paren{\pdv{[\y]} \widetilde{F}(\x_0,[\y_0])}^{-1}
	\]
    \end{linenomath*}
	with the identification $H_0 \cong T_{[\y_0]}(\calY / {\sim})$.
	In addition, $\pdv{\x} F(\x_0,\y_0) = \pdv{\x} \widetilde{F}(\x_0, [\y_0])$. Combining this with \cref{thm: closest solution exists} gives the desired result.
\end{proof}

This proposition adds a new interpretation to \cref{eq: cond quotient manifold}: \emph{the solution map $\calX \to \calY / {\sim}$ has the same condition number as the canonical solution map $\calX \to \calY$}. 
The main advantage of this is that our approach does not require an explicit equivalence relation up to which the solution is defined. That is, one only needs to know that the problem is an FCRE. Moreover, the latent condition number applies to more general problems, as the quotient $\calY/{\sim}$ is not required to be a smooth manifold. Such situations can occur when attempting to quotient by a Lie group that does not act freely; this is exactly what happens when viewing tensor rank decomposition as the problem of recovering factor matrices up to permutation and scaling indeterminacies \cite{Vannieuwenhoven2017}.

One manifold to which \cref{prop: LS is equivalent to quotient} can be applied is the \emph{Grassmannian} of $n$-dimensional linear subspaces of $\mathbb{R}^m$, i.e., $\mathrm{St}(m, n) / O(n)$ where $O(n)$ is the orthogonal group. An equation $\tilde{F}(\x, [\y]) = c$ defining a point $[\y]$ on the Grassmannian can be thought of as an underdetermined system $F(\x,\y) = c$ with outputs on $\mathrm{St}(m, n)$. The condition number \cref{eq: cond quotient manifold} can be obtained by combining \cref{prop: LS is equivalent to quotient} and \cref{thm: closest solution exists}. The same conclusion holds for a general problem over the manifold of positive semidefinite $n \times n$ matrices of rank $k$, which is sometimes identified with $\mathbb{R}^{n \times k}_k / O(k)$ where $\Y_1 \sim \Y_2 \Leftrightarrow \Y_1\Y_1^T = \Y_2\Y_2^T$ \cite{journee_low-rank_2010}.

\section{Condition number of two-factor matrix decompositions}
\label{sec: 2fd}

One of the most basic examples of a FCRE is the factorisation of a matrix into two matrices of full rank. It is formally described as follows.

\begin{definition}\label{def_twofactor}
The \emph{rank-revealing two-factor matrix decomposition problem} at $X \in \mathbb{R}^{m \times n}_k$ is the inverse problem 
$G_\mathcal{M}(L, R) - X = 0$
where 
\[
G_\mathcal{M}: \overbrace{\mathbb{R}^{m \times k}_k \times \mathbb{R}^{k \times n}_k}^{\calY} \to \mathbb{R}^{m \times n},\,
(L,R) \to LR. 
\]
\end{definition}

\begin{proposition}
    \label{prop: cond two-factor decomposition}
    Let $G_\mathcal{M}(L,R) := LR$, where $L \in \mathbb{R}^{m \times k}$ and $R \in \mathbb{R}^{k \times n}$ have rank $k$. Let $\sigma_i(\cdot)$ denote the $i$th largest singular value of its argument if $i \le k$ and $\sigma_i(\cdot) := 0$ if $i > k$.
    At every point $(L,R)$, we have
    \begin{equation}
        \label{eq: cond 2fd}
    \lscond^{inv}[G_\mathcal{M}](L,R)
    =
    \frac{1}{\sqrt{\min\left\{ 
        \sigma_k(L)^2 + \sigma_n(R)^2
        ,\; 
        \sigma_m(L)^2 + \sigma_k(R)^2  
    \right\}}}
    \end{equation}
    with respect to the Euclidean inner product on $\mathbb{R}^{m\times n}$ and $\mathbb{R}^{m\times k}\times\mathbb{R}^{k\times n}$.
    If $k < \min\{m,n\}$, then $\lscond^{inv}[G_\mathcal{M}](L,R) = \min \{\sigma_k(L), \sigma_k(R)\}^{-1}$.
\end{proposition}
\begin{proof}
We will derive the condition number of this problem using \cref{prop: condition of inverse problems}. We can isometrically identify $\mathbb{R}^{m \times n} \cong \mathbb{R}^{mn}$ in the Euclidean distances on both spaces and analogously for $\mathbb{R}^{m \times k}$ and $\mathbb{R}^{k \times n}$. Then, 
\[
DG_\mathcal{M}(L,R)[\dot{L}, \dot{R}] = \dot{L} R + L \dot{R}
\cong
\underbrace{
[\mathds{I}_{m} \otimes R^T \quad L \otimes \mathds{I}_n]
}_{=: J}
\begin{bmatrix}
    \dot{L} \\
    \dot{R}
\end{bmatrix}
.\] 
It remains to compute the $r$th largest singular value of $J$, where $r = \rank(J)$. The singular values of $J$ are the square roots of the eigenvalues of $JJ^T = \mathds{I}_m \otimes (R^T R) + (LL^T) \otimes \mathds{I}_n$. This matrix is a Kronecker sum and its eigenvalues are $\lambda + \mu$ where $\lambda$ and $\mu$ run over all eigenvalues of $R^TR$ and $LL^T$, respectively \cite[Theorem 4.4.5]{horn_topics_2010}. Therefore, all singular values of $J$ are 
\begin{equation}
    \label{eq: svd of two factor decomposition}
\sigma(J) 
=
\left\{ 
\sqrt{\sigma_i(L)^2 + \sigma_j(R)^2}
\quad\middle\vert\quad 
1 \le i \le m,\,
1 \le j \le n
\right\}
.\end{equation}
The number of nonzero singular values of $J$ is thus constant for all $L$ and $R$ of rank $k$ (counted with multiplicity).

By \cref{prop: condition of inverse problems}, the condition number is the reciprocal of the smallest nonzero singular value of $J$. An element of \cref{eq: svd of two factor decomposition} is zero if and only if both $i > k$ and $j > k$. 
Thus, 
\begin{equation}
    \label{eq: 2fd smallest svd Jacobian}
\lscond^{inv}[G_\mathcal{M}](L,R) = \left(\min_{i \le k \text{ or } j \le k} \sqrt{\sigma_i(L)^2 + \sigma_j(R)^2} \right)^{-1}
.\end{equation}
Since the singular values are sorted in descending order, the minimum is attained when $i = k$ or $j = k$. If it is attained for $i = k$, the right-hand side of \cref{eq: 2fd smallest svd Jacobian} is $(\sigma_k(L)^2 + \sigma_n(R)^2)^{-1/2}$. Analogously, if the minimum is attained for $j = k$, we get $(\sigma_m(L)^2 + \sigma_k(R)^2)^{-1/2}$. This concludes the general case. The expression for the case where $k < \min\{m,n\}$ is obtained by substituting $\sigma_m(L) = \sigma_n(R) = 0$ in \cref{eq: cond 2fd}.
\end{proof}

Not all two-factor decompositions of a given matrix $\X \in \mathbb{R}^{m \times n}_k$ have the same condition number. Therefore, one may be interested in a decomposition whose condition number is as small as possible. In the context of tensor decompositions, the \emph{norm-balanced CPD} was introduced for the same purpose \cite{Vannieuwenhoven2017}.
Intuitively, one may expect to find an optimal two-factor decomposition by computing a singular value decomposition $\X = U \Sigma V^T$ and setting $L := U\Sigma^{1/2}$ and $R := \Sigma^{1/2} V^T$. This turns out to be correct, by the following lemma.

\begin{lemma}
    Suppose that $\X = LR$ with $L \in \mathbb{R}^{n \times k}_k$ and $R \in \mathbb{R}^{k \times n}_k$. Then $\min\left\{\sigma_k(L), \sigma_k(R)\right\} \le \sqrt{\sigma_k(\X)}$.
\end{lemma}
\begin{proof}
    Let $U$ and $V$ be matrices whose columns are orthonormal bases of $\operatorname{span} X$ and $\operatorname{span} X^T$, respectively. If we set $(\widehat{L}, \widehat{R}, \widehat{X}) := (U^T L, RV, U^TXV)$, then the $k\times k$ matrices $\widehat{L}, \widehat{R},$ and $\widehat{X}$ have the same $k$ largest singular values as $L, R$, and $X$, respectively.
    Suppose that $\sigma_k(\widehat{X}) = \norm{\widehat{X}v}$ for some unit vector $v \in \mathbb{R}^k$, then $\sigma_k(\widehat{X}) = \bigl\Vert \widehat{L}\widehat{R}v \bigr\Vert \ge \sigma_k(\widehat{L})\sigma_k(\widehat{R})$ by the Courant--Fisher theorem \cite[Theorem 3.1.2]{horn_topics_2010}. Hence, $\sigma_k(\widehat{L})$ and $\sigma_k(\widehat{R})$ cannot both be larger than $\sqrt{\sigma_k(X)}$.
\end{proof}
\begin{corollary}
    \label{cor: optimal condition number 2fd}
    Let ${\X_0} \in \mathbb{R}^{m \times n}_k$ be any matrix and let $G_\mathcal{M}$ be the map from \cref{prop: cond two-factor decomposition}. Then, the best latent condition number of computing a two-factor matrix factorisation is
    \[
    \min_{{L_0R_0 = \X_0}} \lscond^{inv}[G_\mathcal{M}]({L_0, R_0}) = \sigma_k({\X_0})^{-1/2}
    .\]
    If ${\X_0} = U\Sigma V^T$ is a compact singular value decomposition\footnote{We call a singular value decomposition \emph{compact} if $\Sigma \in \mathbb{R}^{k \times k}$ and $k = \rank X_0$.}, then the minimum is attained at ${(L_0,R_0)} = (U \Sigma^{1/2}, \Sigma^{1/2} V^T)$.
\end{corollary}

\begin{remark}
    \Cref{cor: optimal condition number 2fd} should \emph{not} be interpreted as saying that the evaluation of the map $X \mapsto (U\Sigma^{1/2}, \Sigma^{1/2}V^T)$ is well-conditioned or that it refines the two-factor decomposition optimally in the sense defined in the introduction. This is clearly false, since the singular vectors of $X$ may not even be unique. Instead, \cref{cor: optimal condition number 2fd} says that, if one is interested in a two-factor decomposition $(L_0,R_0)$ of $X_0$ such that rank-preserving perturbations $X$ of $X_0$ have \emph{any} decomposition close to $(L_0, R_0)$, then $(U \Sigma^{1/2}, \Sigma^{1/2} V^T)$ is optimal.
\end{remark}

\Cref{cor: optimal condition number 2fd} connects the condition number to the distance from $\X \in \calX = {\mathbb{R}^{m \times n}_k}$ to the boundary {$\partial X$} of $\calX$. Since $\partial\calX$ is the set of $m \times n$ matrices of rank strictly less than $k$, the Eckart-Young theorem implies that $\min_{\widehat{\X} \in \partial\calX} \norm{\X - \widehat{\X}}_2 = \sigma_k(\X)$, which is the inverse square of the condition number in \cref{cor: optimal condition number 2fd}. Consequently, the \emph{ill-posed locus}, defined as the set of (limits of) inputs where the condition number diverges, is precisely the boundary of $\mathcal{X}$.
For many numerical problems, there is a connection between the condition number and the reciprocal distance to the ill-posed locus, often called a \emph{condition number theorem} \cite{demmel_condition_1987, BCSS}. The above shows that the two-factor decomposition admits such a connection as well.

\section{Condition number of orthogonal Tucker decompositions}
\label{sec: Tucker}

As another application of the proposed theory,
in this section, we study the condition number of a different rank-revealing decomposition, this time in the context of tensors.
Given a tensor 
\[
\mathpzc{S} = \sum_{i=1}^R {s}_{i,1} \otimes {s}_{i,2} \otimes \dots \otimes {s}_{i, D} \in \mathbb{R}^{k_1\times k_2\times\cdots\times k_D},
\]
where $s_{i,j} \in \mathbb{R}^{k_j}$ are vectors, 
the tensor product\footnote{In case of unfamiliarity with the tensor product, all occurrences of $\otimes$ may be interpreted as the Kronecker product of matrices and vectors.} of the matrices $U_j \in \mathbb{R}^{n_j \times k_i}, j=1,\ldots,D,$ acts linearly on $\mathpzc{S}$ as 
\begin{equation}
    \label{eq: Tucker product}
    \xT = 
(U_1 \otimes \dots \otimes U_D) \mathpzc{S} = \sum_{i=1}^R \left(U_1{s}_{i, 1}\right) \otimes \left(U_2 {s}_{i,2}\right) \otimes \dots \otimes \left(U_D {s}_{i,D}\right)
.\end{equation}
The resulting tensor $\xT$ lives in $\mathbb{R}^{n_1\times n_2\times\cdots\times n_D}$.
When $D=2$, $\mathpzc{S}$ is a matrix and the above expression can be simplified to $U_1 \mathpzc{S} U_2^T$. The \emph{Tucker decomposition problem} \cite{tucker1966some} takes a tensor $\xT$ as in \cref{eq: Tucker product} and asks to recover the factors $U_1, \dots, U_D, \mathpzc{S}$. It is common to impose that all columns of $U_i$ are orthonormal for each $i$, in which case \cref{eq: Tucker product} is sometimes called an \emph{orthogonal Tucker decomposition}.  

To formulate the problem more precisely, we introduce some notation.
For the above tensor $\mathpzc{S}$, the $j$th flattening is the matrix
\[
\mathpzc{S}_{(j)} = \sum_{i=1}^R {s}_{i,j} \cdot \mathrm{vec}(s_{i,1} \otimes \dots \otimes s_{i,j-1} \otimes s_{i,j+1} \otimes \dots \otimes s_{i,D})^T
,\]
where $\mathrm{vec}(x_1 \otimes \dots \otimes x_D)$ is the Kronecker product of the vectors $x_1, \dots, x_D$. If $\mathpzc{S}$ is a matrix (i.e., $D=2$), then $\mathpzc{S}_{(1)} = \mathpzc{S}$ and $\mathpzc{S}_{(2)} = \mathpzc{S}^T$.
The \emph{multilinear rank} of $\mathpzc{S}$ is the tuple $\mu(\mathpzc{S}) := (\rank \mathpzc{S}_{(1)},\dots,\rank \mathpzc{S}_{(D)})$. We say that $\mathpzc{S}$ has \emph{full} multilinear rank if $\mu(\mathpzc{S}) = (k_1,\dots,k_D)$. The set of such tensors is written as $\mathbb{R}^{k_1 \times \dots \times k_D}_\star$. 

\begin{definition}
    \label{def: orth Tucker}
    The \emph{rank-revealing orthogonal Tucker decomposition problem} at $\mathpzc{X}\in\mathbb{R}^{n_1\times\cdots\times n_D}$ is the inverse problem $\tuck(U_1,\dots,U_D,\mathpzc{S}) - \xT = 0$ where 
    \begin{linenomath*}
    \begin{align*}
        \tuck:\, \overbrace{\mathrm{St}(n_1, k_1) \times \dots \times \mathrm{St}(n_D, k_D) \times \mathbb{R}^{k_1 \times \dots \times k_D}_\star}^{\calY} &\to \mathbb{R}^{n_1 \times \dots \times n_D} \\
        (U_1, \dots, U_D, \mathpzc{S}) &\mapsto (U_1 \otimes \dots \otimes U_D) \mathpzc{S}
    .\end{align*}
    \end{linenomath*}
\end{definition}
The reason for considering $\mathbb{R}^{k_1 \times \dots \times k_D}_\star$ in the domain rather than its closure $\mathbb{R}^{k_1 \times \dots \times k_D}$ is to ensure that $\rank D\tuck$ is constant \cite{Koch2010}. 
If $\y = (U_1,\dots,U_D, \mathpzc{S})$ solves the orthogonal Tucker decomposition problem, then all other solutions can be parameterised as
\begin{equation}
    \label{eq: fibres Tucker}
\tuck^{-1}(\tuck(y)) = \left\{ 
   (U_1 Q_1^T, \dots, U_D Q_D^T, (Q_1 \otimes \dots \otimes Q_D) \mathpzc{S}) 
   \mid
   Q_j \in O(k_j)
\right\}
.\end{equation}
In the literature on multi-factor principal component analysis, these invariances are sometimes called ``rotational'' degrees of freedom \cite{kiers_three-way_2001}.

To eliminate some degrees of freedom, it has been proposed to impose constraints on the core tensor $\mathpzc{S}$. For instance, one could optimise a measure of sparsity on $\mathpzc{S}$ to enhance the interpretability of the decomposition \cite{kiers_three-way_2001,martin_jacobi-type_2008}. Alternatively, the \emph{higher-order singular value decomposition} (HOSVD) \cite{DeLathauwer2000} imposes pairwise orthogonality of the slices of $\mathpzc{S}$. The HOSVD has the advantages of being definable in terms of singular value decompositions and giving a quasi-optimal solution to Tucker approximation problems \cite[Theorem 10.3]{hackbusch_tensor_2012}. It is unique if and only if the singular values of all $\xT_{(i)}$ are simple for all $i$. 

However, for these constrained Tucker decompositions, important geometric properties of the set of feasible values of $\mathpzc{S}$ remain elusive. For the HOSVD, it remains unknown precisely what sets of singular values of the flattenings $\mathpzc{S}_{(i)}$ are feasible \cite{hackbusch_interconnection_2017}. Furthermore, for two tensors $\mathpzc{S}, \mathpzc{S}'$ with HOSVD constraints, the singular values of $\mathpzc{S}_{(i)}$ and $\mathpzc{S}'_{(i)}$ may be identical for all $i$ even if $\mathpzc{S}$ and $\mathpzc{S}'$ are in distinct $O(k_1) \times \dots \times O(k_D)$-orbits \cite{hackbusch_interconnection_2017}.
For these reasons, we ignore the constraints to make the orthogonal Tucker decomposition (usually) unique and study the underdetermined problem of \cref{def: orth Tucker} instead.

To determine the condition number, we need a Riemannian metric for the domain and codomain of $\tuck$. A simple metric is the Euclidean or Frobenius inner product, which is defined on (the tangent spaces of) $\mathbb{R}^{n_1 \times \dots \times n_D}$, $\mathbb{R}^{k_1 \times \dots \times k_D}_\star$, and $\mathrm{St}(n_i, k_i) \subset \mathbb{R}^{n_i \times k_i}$. Thus, we may use the associated product metric for $\calY$. We call this metric of $\calY$ and the Euclidean inner product in $\mathbb{R}^{n_1 \times \dots \times n_D}$ \emph{absolute (Riemannian) metrics}. The norm induced by these metrics is the Euclidean or Frobenius norm, which we denote by $\norm{\cdot}_F$.

Since the Stiefel manifold is bounded in the Euclidean metric and $\mathbb{R}^{k_1 \times \dots \times k_D}_\star$ is not, it may be more interesting to work with \emph{relative} metrics.
For a punctured Euclidean space $\mathbb{E} \backslash \{0\}$ with inner product $\langle\cdot,\cdot\rangle$, the relative metric for two vectors $\xi, \eta \in T_{p} \mathbb{E}$ is $\frac{\langle\xi,\eta\rangle}{\langle p,p\rangle}$. Note that this defines a smooth Riemannian metric.
We lift this definition so that the relative metric in $\calY$ is the product metric of the relative metric in $\mathbb{R}^{k_1 \times \dots \times k_D}_\star$ and the Frobenius inner products on all $\mathrm{St}(n_i, k_i)$.

\begin{proposition}
    \label{prop: condition Tucker}
    Let $(U_1 \otimes \dots \otimes U_D) \mathpzc{S}$ be an orthogonal Tucker decomposition of a tensor $\xT \in \mathbb{R}^{n_1 \times \dots \times n_D}$ such that $\mathpzc{S} \in \mathbb{R}^{k_1 \times \dots \times k_D}_\star$ and $k_i < n_i$ for at least one $i$.
    Let $\sigma := \min_{i:\, k_i < n_i} \sigma_{k_i}(\mathpzc{S}_{(i)})$.
    Then,
    \begin{enumerate}
        \item $\lscond^{inv}[\tuck](U_1,\dots,U_D,\mathpzc{S}) = \max \left\{ \frac{1}{\sigma}, 1 \right\}$ for the absolute metric, and\vspace{3pt}
        \item $\lscond^{inv}[\tuck](U_1,\dots,U_D,\mathpzc{S}) = \frac{\norm{\xT}_F}{\sigma}$ for the relative metric.
    \end{enumerate}
\end{proposition}
\begin{proof}
Throughout this proof, we abbreviate $D\tuck(U_1,\dots,U_D,\mathpzc{S})$ to $D\tuck$.
Following \cref{prop: condition of inverse problems}, we compute the smallest nonzero singular value of $D\tuck$ for both metrics. The proof consists of a construction of this matrix with respect to an orthonormal basis and a straightforward calculation of its singular values.
To use the same derivation for the two metrics, we let $\alpha = 1$ for the absolute metric and $\alpha = \norm{\mathpzc{S}}_F = \norm{\xT}_F$ for the relative metric.

For all $i$, let $\{\Omega_i^j \mid 1 \le j \le \frac{1}{2}k_i (k_i - 1) \}$ be an orthonormal basis of the $k_i \times k_i$ skew-symmetric matrices. Let $U_i^\perp \in \mathrm{St}(n_i, n_i - k_i)$ be any matrix so that $[U_i \quad U_i^\perp]$ is orthogonal. If $n_i = k_i$, we write $U_i^\perp$ formally as an $n_i \times 0$ matrix. Let $\{V_i^p \mid 1 \le p \le (n_i - k_i)k_i \}$ be a basis of $\mathbb{R}^{(n_i - k_i) \times k_i}$. Then, $\mathcal{B}_i := \{ U_i \Omega_i^j + U_i^\perp V_i^p \}$ forms an orthonormal basis of $T_{U_i}\mathrm{St}(n_i, k_i)$. If $\{E_l \,\vert\, 1 \le l \le \prod_{i=1}^D k_i\}$ is the canonical basis of $\mathbb{R}^{k_1 \times \dots \times k_D}$, then $\mathcal{B}_0 := \{\alpha E_l\}$ is an orthonormal basis of $\mathbb{R}^{k_1 \times \dots \times k_D}$. The product of these bases gives a canonical orthonormal basis for $\calY$.

Similarly, an orthonormal basis of $\mathbb{R}^{n_1 \times \dots \times n_D}$ is $\{\alpha \hat{E}_j\}$ where the $\hat{E}_j$ are the canonical basis vectors of $\mathbb{R}^{n_1 \times \dots \times n_D}$. In other words, expressing a vector in orthonormal coordinates is equivalent to division by $\alpha$.

Next, we compute the differential of $\tuck$. For general tangent vectors $\dot{\mathpzc{S}} \in T_{\mathpzc{S}}\mathbb{R}^{k_1 \times \dots \times k_D}_\star$ and $\dot{U}_i \in T_{U_i}\mathrm{St}(n_i, k_i)$, we have, in coordinates,
\begin{linenomath*}
\begin{align*}
D\tuck[0,\dots,0,\dot{\mathpzc{S}}] &= \alpha^{-1}
(U_1 \otimes \dots \otimes U_D) \dot{\mathpzc{S}} \quad \text{and} \\
D\tuck[0,\dots,\dot{U}_i,\dots,0] &=  \alpha^{-1} (U_1 \otimes \dots \otimes U_{i-1} \otimes \dot{U}_i \otimes U_{i+1} \otimes \dots \otimes U_D) \mathpzc{S}
,\end{align*}
\end{linenomath*}
which is extended linearly for all tangent vectors.
The condition that $U_i^T U_i^\perp = 0$ for all $i$ splits the image of $DG$ into pairwise orthogonal subspaces. That is, for any $i$ and $k$, $DG[0,\dots,U_i^\perp V_i^p,\dots,0]$ is orthogonal to both $D\tuck[0,\dots,0,\dot{\mathpzc{S}}]$ for all $\dot{\mathpzc{S}}$ and $D\tuck[0,\dots,\dot{U}_{i'},\dots,0]$ for all $\dot{U}_{i'}$ where $i' \ne i$.

To decompose the domain of $D\tuck$ as a direct sum of pairwise orthogonal subspaces, we write $T_{U_i} \mathrm{St}(n_i, m_i) = W_i \oplus W_i^\perp$ for all $i$, where $W_i := \mathrm{span}\{U_i \Omega_i^j\}$ and $W_i^\perp = \mathrm{span}\{U_i^\perp V_i^p\}$. The restriction of $D\tuck$ to $W_1 \times \dots \times W_D \times T_{\mathpzc S} \mathbb{R}^{k_1 \times \dots \times k_D}_\star$ can be represented in coordinates by a matrix $J_0$. Likewise, for $i = 1,\dots, D$, we write the restriction of $D\tuck$ to $\{0\} \times \dots \times W_i^\perp \times \dots \times \{0\}$ in coordinates as $J_i$. Then $D\tuck$ can be represented as $J = \begin{bmatrix}J_0 & J_1 & \dots & J_D\end{bmatrix}$.

By the preceding argument, these $D+1$ blocks that make up $J$ are pairwise orthogonal. Therefore, the singular values of $J$ are the union of the singular values of $J_0,\dots,J_{D}$. If $n_i = k_i$ for some $i$, then $J_i$ is the matrix with zero columns, whose singular values are the empty set.

Let $\Pi := \prod_{i=1}^D k_D$.
To bound the singular values of $J_0$, we compute its first $\Pi$ columns as
$
D\tuck[0,\dots,0,\alpha E_l] = (U_1 \otimes \dots \otimes U_D) E_l
$
for all $l = 1,\dots,\Pi$.
Note that all other columns have a factor $(U_1 \otimes \dots \otimes U_D)$ as well. 
Thus, we have 
\[
J_0 = (U_1 \otimes \dots \otimes U_D) \left[ \mathds{I}_{\Pi} \quad  \tilde{J} \right]
,\]
where $\tilde{J}$ is an unspecified matrix. We can omit the orthonormal factor $U_1 \otimes \dots \otimes U_D$ when computing the singular values of $J_0$. Therefore, $J_0$ has $\Pi$ nonzero singular values, which are the square roots of the eigenvalues of $\mathds{I}_{\Pi} + \tilde{J}\tilde{J}^T$. It follows that the $\Pi$ largest singular values of $J_0$ are bounded from below by $1$ and all other singular values of $J_0$ are $0$.

Next, consider any $J_i$ where $i \ge 1$ and $n_i > k_i$. It represents the linear map 
\[
J_i:\, V \mapsto \alpha^{-1} (U_1 \otimes \dots \otimes U_{i - 1} \otimes U_i^\perp V \otimes U_{i+1} \otimes \dots \otimes U_D) \mathpzc{S}
.\]
Up to reshaping, the above is equivalent to 
\[
J_i:\, \operatorname{vec} V \mapsto \alpha^{-1} \left(U_i^\perp \otimes ((U_1 \otimes \dots \otimes U_{i-1} \otimes U_{i+1} \otimes \dots \otimes U_D) \mathpzc{S}_{(i)}^T)\right) \operatorname{vec} V
.\]
To calculate the singular values of $J_i$, we factor out $U_i^\perp$ and all $U_j$ to obtain $J_i \cong \alpha^{-1} \mathds{I}_{n_i - k_i} \otimes \mathpzc{S}_{(i)}^T$. Its singular values are the singular values of $\alpha^{-1} \mathpzc{S}_{(i)}$ with all multiplicities multiplied by $n_i - k_i$.

Finally, we show geometrically that the smallest nonzero singular value $\varsigma$ of $D\tuck$ is at most 1. By the Courant-Fisher theorem, $\varsigma \le \norm{D\tuck[\xi]} / \norm{\xi}$ for all $\xi \in \ker(D\tuck)^\perp \backslash \{0\}$.
Pick $\xi := (0,\dots,0,\mathpzc{S})$.
Since 
\[
\tuck^{-1}(\xT) = \left\{ 
    (U_1 Q_1^T,\dots,\,U_D Q_D^T, (Q_1 \otimes \dots \otimes Q_D) \mathpzc{S})
        \mid 
    Q_i \in O(k_i)
    \right\}
,\]
the projection of $\tuck^{-1}(\xT)$ onto the last component is a submanifold of the sphere over $\mathbb{R}^{k_1 \times \dots \times k_D}$ of radius $\norm{\mathpzc{S}}_F$. It follows that
\[
\xi \in N_{(U_1, \dots, U_D, \mathpzc{S})}\tuck^{-1}(\xT) = (\ker D\tuck)^\perp,
\]
where $N$ denotes the normal space.
Since $D\tuck[\xi] = (U_1 \otimes \dots \otimes U_D) \mathpzc{S}$, we obtain $\varsigma \le \norm{D\tuck[\xi]} / \norm{\xi} = 1$.

In conclusion, we have established the following three facts. First, for all $i$ such that $k_i < n_i$, all singular values of $\alpha^{-1} \mathpzc{S}_{(i)}$ are singular values of $D\tuck$. Second, any other nonzero singular values of $D\tuck$ must be bounded from below by 1. Third, the smallest nonzero singular value of $D\tuck$ is at most 1. 
By \cref{prop: condition of inverse problems}, this proves the statement about the absolute metric.
For the relative metric, note that $\sigma \le \norm{\mathpzc{S}}_F = \norm{\xT}_F = \alpha$ for all $i$. Thus, the smallest nonzero singular value of $D\tuck$ is $\sigma / \alpha = \sigma/\norm{\xT}_F \le 1$.
\end{proof}

The expression for the absolute metric can be interpreted as follows: if $k_i < n_i$ and $\sigma_{(i)} := \sigma_{k_i}(\mathpzc{S}_{(i)})$ is small, then the factor $U_i$ is sensitive to perturbations of $\xT$.
Indeed, assume that the last column of $U_i$ is a left singular vector corresponding to the singular value $\sigma_{(i)}$. Then, we can generate the following small perturbation of $\xT$ that corresponds to a unit change in the decomposition. Let $\tilde{U}_i$ be a matrix such that $U_i e_j = \tilde{U}_i e_j$ for $1 \le j < k_1$ and $U_i^T \tilde{U}_i e_{k_i} = 0$. 
The perturbed tensor $\widetilde{\xT} = \tuck(U_1, \dots, U_{d-1}, \tilde{U}_i, U_{d+1}, \dots, U_D, \mathpzc{S})$ is only at a distance $\sigma_{(i)}$ away from $\xT$.

On the other hand, if $\sigma \ge 1$, then no unit perturbation of $\xT$ tangent to $\tuck(\calY)$ can change the orthogonal Tucker decomposition more than the tangent vector $\Delta \xT = \xT / \norm{\xT}_F$ constructed in the proof of \cref{prop: condition Tucker}. 

\begin{remark}[Condition number of a singular value decomposition]
    For a matrix $\X$, computing an orthogonal Tucker decomposition of the form $\X = U_1 S U_2^T$ is a relaxation of the singular value decomposition that does not impose a diagonal structure on $S$.
    A condition number for the subspaces spanned by the singular vectors was studied in \cite{sunPerturbationAnalysisSingular1996,vannieuwenhoven_condition_2023}.
    The condition number for the $i$th singular vector diverges as $\left\vert\sigma_{i}(\X) - \sigma_{i+1}(\X)\right\vert \to 0$. By contrast, the condition number of computing an orthogonal Tucker decomposition depends on $\sigma_{k_1}(\X)$ and $\norm{\X}_F$ only.
    Thus, computing individual singular vectors may be arbitrarily ill-conditioned even if the condition number of the Tucker decomposition is arbitrarily close to one. For instance, this occurs for the $3 \times 3$ diagonal matrix $X$ whose diagonal elements are $(1 + \varepsilon, 1, 0)$ and $0 < \varepsilon \ll 1$.
    Informally, this observation shows that
    restricting $S$ to be diagonal in an orthogonal Tucker decomposition can make computing the resulting singular value decomposition arbitrarily more ill-conditioned than computing an orthogonal Tucker decomposition.
\end{remark}

\section{Numerical verification of the error estimate}
\label{sec: experiments}

\Cref{eq: 1st order error bound LS} is only an asymptotic estimate of the optimal forward error.
A common practice for working with condition numbers is to neglect the asymptotic term $o(d_{\calX}(\x_0, \x))$ and turn \cref{eq: 1st order error bound LS} into the approximate upper bound 
\begin{equation}
    \label{eq: 1st order err bound approx}
\min_{\substack{\y \in \calY,\\ F(\x,\y) = c}} d_{\calY}(\y_0, \y)
\lesssim 
\lscond[F^{-1}(c)](\x_0, \y_0) \cdot d_{\calX}(\x_0, \x)
.\end{equation}
In this section, we determine numerically if this approximation is accurate for random initial solutions $(\x_0, \y_0)$ and random perturbations $\x$. We restrict ourselves to the orthogonal Tucker decomposition of third-order tensors.

\subsection{Model}
The model is defined as follows. We pick a parameter $\alpha > 0$ to control the condition number. We generate matrices $A, B \in \mathbb{R}^{k \times (k - 1)}$ and a tensor $\mathpzc{H} \in \mathbb{R}^{k \times k \times k}$ with i.i.d.\ standard normally distributed entries. Then, we set $\mathpzc{S}' \in \mathbb{R}^{k \times k \times k}$ so that $\mathpzc{S}'_{(1)} = (AB^T + \alpha \mathds{I})\mathpzc{H}_{(1)}$ and normalise $\mathpzc{S}_{0} := \mathpzc{S'} / \norm{\mathpzc{S'}}_F$. Finally, we generate the factor $U_i^0$ by taking the $Q$-factor in the QR decomposition of a normally distributed $n_i \times k$ matrix, for $i=1,2,3$.
To the above matrices and tensor, we associate the Tucker decomposition $\mathpzc{X}_0 = \tuck(U_1^0,U_2^0,U_3^0,\mathpzc{S}_{0})$. In the following, we abbreviate $\kappa := \lscond^{inv}[\tuck](U_1^0, U_2^0, U_3^0, \mathpzc{S}_0)$. Note that choosing a small value of $\alpha$ generally makes the problem ill-conditioned by \cref{prop: condition Tucker}.

The condition number measures the change to the decomposition for \emph{feasible} perturbations of $\xT_0 \in \mathbb{R}^{n_1\times n_2\times n_3}$. That is, the perturbed input is a point $\xT$ in the image of $\tuck$. Such a point can be generated by first perturbing $\xT_0$ in the ambient space $\mathbb{R}^{n_1 \times n_2 \times n_3}$ as $\xT' := \xT_0 + \varepsilon \Delta\xT$ where $\Delta\xT$ is uniformly distributed over the unit-norm tensors in $\mathbb{R}^{n_1\times n_2\times n_3}$ and $\varepsilon > 0$ is some parameter.
Then, a feasible input can be obtained by applying the ST-HOSVD algorithm \cite{Vannieuwenhoven2012} to $\xT'$ with truncation rank $(k,k,k)$. This gives a quasi-optimal projection $\xT$ of $\xT'$ onto the image of $\tuck$ and a decomposition $\xT = \tuck(U_1, U_2, U_3, \mathpzc{S})$. 

\subsection{Estimate of the optimal forward error}

Since $\norm{\xT_0}_F = \norm{\mathpzc{S}_0}_F = 1$, the absolute and relative metric in \cref{prop: condition Tucker} coincide, and they both correspond to the product of the Euclidean inner products in $\mathbb{R}^{k_1 \times k_2 \times k_3}$ and all $\mathrm{St}(n_i, k_i)$. By \cref{eq: fibres Tucker}, the orthogonal Tucker decomposition of $\xT$ is determined up to a multiplication by $Q_1, Q_2, Q_3 \in O(k)$. Thus, the square of the left-hand side of \cref{eq: 1st order err bound approx} is\footnote{
    Although the definition of the condition number uses the geodesic distance $d_\gamma$, the following expression uses the Euclidean distance $d_E$. However, it can be shown that for a point $\x_0 \in \calX$ where $\calX$ is  a Riemannian submanifold of Euclidean space, we have $d_E(x_0, x) = d_\gamma(x_0, x)(1 + o(1))$ as $x \to x_0$.
} 
\begin{equation}
    \label{eq: optimisation distance Tucker}
E^2 := 
    \min_{Q_1,Q_2, Q_3 \in O(k)} \left\{ 
        \norm{\mathpzc{S}_0 - (Q_1^T \otimes Q_2^T \otimes Q_3^T) \mathpzc{S}}_F^2
        +
        \sum_{i = 1}^3 \norm{U_i^{0} - U_i Q_i}_F^2
    \right\}
.\end{equation}
Determining the accuracy of \cref{eq: 1st order err bound approx} requires evaluating $E$ numerically.
Since we are not aware of any closed-form expression of $E$, we approximate this by solving the above optimisation problem using a simple Riemannian gradient descent method in \texttt{Manopt.jl} \cite{bergmann_manoptjl_2022}. In the following, $\widehat{E}^2$ denotes the numerical solution to \cref{eq: optimisation distance Tucker}.

Assuming that $\widehat{E}$ is an accurate approximation of $E$, we check \cref{eq: 1st order err bound approx} by verifying that $\widehat{E} \lesssim \kappa \norm{\xT - \xT_0}_F$. A priori, there are at least two scenarios in which this may fail to be a tight upper bound:
\begin{enumerate}
\item The tolerance of the gradient descent method that computes $\widehat{E}$ is $5 \times 10^{-8}$, so that the numerical error in computing $E$ is about the same order of magnitude. If $E \ll 10^{-8}$, then $\widehat{E}$ is probably a poor approximation of $E$.
\item $E$ cannot be much larger than $1$. Since $\mathrm{St}(n_i, k)$ is a subset of the $n_i\times k$ matrices with Frobenius norm equal to $\sqrt{k}$, we have $\norm{U_i^0 - U_i}_F \le 2\sqrt{k}$. Furthermore, since $\norm{\mathpzc S}_F = \norm{\xT}_F$ and $\norm{\mathpzc S_0}_F = \norm{\xT_0}_F = 1$, it follows from the triangle inequality that $\norm{\mathpzc S - \mathpzc S_0}_F \le 1 + \norm{\xT}_F$. Thus, if $\kappa \norm{\xT - \xT_0}_F \gg 1$, this would overestimate $E$. 
\end{enumerate}
For these reasons, we are only interested in verifying the estimate $\widehat{E} \lesssim \kappa \norm{\xT - \xT_0}_F$ if $\widehat{E} \ge 5 \times 10^{-8}$ and $\kappa\norm{\xT - \xT_0}_F \le 1$.

\subsection{Experimental results}
We generated two datasets as specified by the model above. In the first dataset, we used the parameters $k = 3$ and $(n_1, n_2, n_3) = (5,5,5)$. For each pair $(\alpha, \varepsilon) \in \{10^{-8}, 10^{-4}, 1\} \times \{10^{-14}, 10^{-12.5},\dots,10^{-2}\}$, we generated 2000 Tucker decompositions and perturbations and measured the error. The second dataset was generated the same way, the only difference being that $n_1 = 2000$.

Since the condition number depends only on $\mathpzc{S}_0$, the distribution of the condition number is the same for both datasets. We found that $\kappa$ is approximately equal to $10 / \alpha$, with $1 \le \kappa\alpha \le 100$ in $94.5\%$ of samples. The empirical geometric mean of $\kappa\alpha$ is about $12$. This means that we can roughly control the condition number of the sampled tensor by controlling the parameter $\alpha$.

\Cref{fig: error distribution} shows the distribution of $\widehat{E}/ \kappa \norm{\xT - \xT_0}_F$ for both datasets.
The smaller this quantity, the more pessimistic $\kappa \norm{\xT - \xT_0}_F$ is as an estimate of the forward error for random perturbed tensors $\mathcal{X}$.
In most cases displayed on \cref{fig: error distribution}, $\widehat{E}$ is at least a fraction $0.1$ of its approximate upper bound $\kappa \norm{\xT - \xT_0}_F$. In the case where $2000 = n_1 \gg k = 3$, we have $\widehat{E} \approx 0.5 \kappa \norm{\xT - \xT_0}_F$. These experiments indicate that the approximation $E \lesssim \kappa \norm{\xT - \xT_0}$ is reasonably sharp on average.

The estimate $E \lesssim \kappa \norm{\xT - \xT_0}$ could be sharpened by using a \emph{stochastic condition number}, which estimates the forward error corresponding to uniform random perturbations $\xT$ on a sphere around $\xT_0$ rather than worst-case perturbations. It was shown in \cite{armentanoStochasticPerturbationsSmooth2010} that the stochastic condition number of a map $H: \mathcal X \to \mathcal Y$ could be as low as $\kappa / \sqrt{\dim \mathcal X}$ where $\kappa$ is the condition number. The idea is that there may only be one bad perturbation direction in $\mathcal{X}$, which is unlikely to manifest in practice if $\mathcal{X}$ is high-dimensional. In our case, though, the error is empirically closer to the worst case in the high-dimensional experiment $(n_1 = 2000)$ than in the low-dimensional one $(n_1 = 5)$. This is evidence that, for the decomposition of large tensors of low multilinear rank, the ill-conditioned perturbation directions fill up more of the space. A full stochastic analysis is beyond the scope of this paper.

\begin{figure}[h]
    \centering 
        \includegraphics[width=0.95\textwidth]{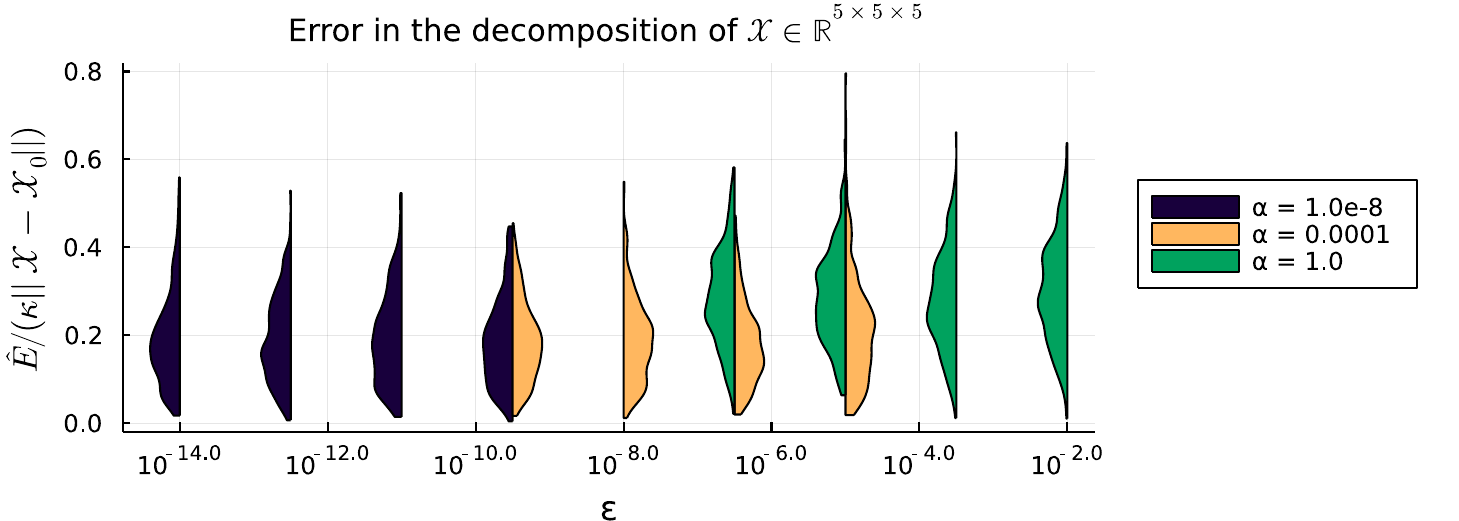}
        \includegraphics[width=0.95\textwidth]{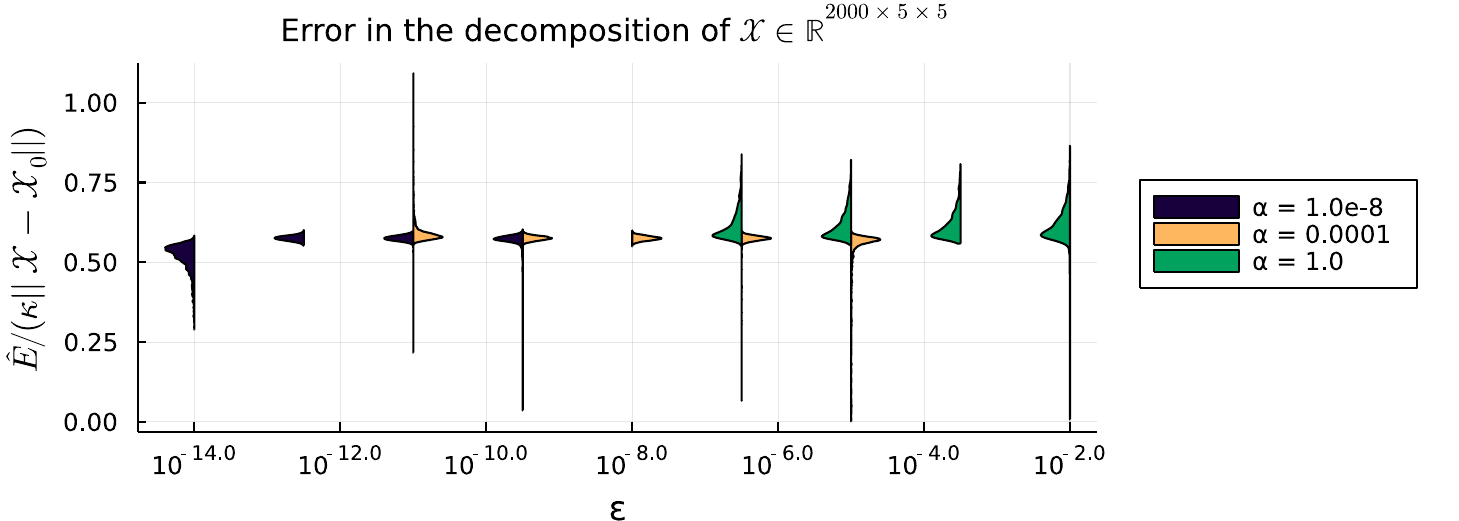}
    \caption{Distribution of $\frac{\widehat{E}}{\kappa \norm{\xT - \xT_0}_F}$ for a Tucker decomposition of the perturbed tensor $\xT$. An estimate of the probability density is plotted for all combinations of $(\alpha, \varepsilon)$ such that $90\%$ of samples satisfy $\widehat{E} \ge 5 \times 10^{-8}$ and $90\%$ of samples satisfy $\kappa \norm{\xT - \xT_0}_F \le 1$.}
    \label{fig: error distribution}
\end{figure}

\section{Acknowledgements}
We thank three anonymous reviewers and editor Ilse Ipsen, whose comments helped greatly with the presentation and clarity of our results. Most content of sections \ref{sec: intro} and \ref{sec: statement of main results} was rewritten, and \cref{sec: prelims} was added, to incorporate their suggested clarifications.

\section{Statements and Declarations}
The authors are not aware of any competing interests.

% \bibliographystyle{abbrv}
% \bibliography{underdetermined}

\begin{thebibliography}{10}

\bibitem{armentanoStochasticPerturbationsSmooth2010}
D.~Armentano.
\newblock Stochastic perturbations and smooth condition numbers.
\newblock {\em Journal of Complexity}, 26(2):161--171, Apr. 2010.

\bibitem{Arslan2019}
B.~Arslan, V.~Noferini, and F.~Tisseur.
\newblock The structured condition number of a differentiable map between
  matrix manifolds, with applications.
\newblock {\em SIAM Journal on Matrix Analysis and Applications},
  40(2):774--799, 2019.

\bibitem{bergmann_manoptjl_2022}
R.~Bergmann.
\newblock Manopt.jl: {Optimization} on {manifolds} in {Julia}.
\newblock {\em Journal of Open Source Software}, 7(70):3866, Feb. 2022.

\bibitem{BCSS}
L.~Blum, F.~Cucker, M.~Shub, and S.~Smale.
\newblock {\em {Complexity and Real Computation}}.
\newblock Springer-Verlag, New York, 1998.

\bibitem{Breiding2018a}
P.~Breiding and N.~Vannieuwenhoven.
\newblock The {condition} {number} of {join} {decompositions}.
\newblock {\em SIAM Journal on Matrix Analysis and Applications},
  39(1):287--309, Jan. 2018.

\bibitem{Breiding2019}
P.~Breiding and N.~Vannieuwenhoven.
\newblock The {condition} {number} of {Riemannian} {aproximation} {Problems}.
\newblock {\em SIAM Journal on Optimization}, 31(1):1049--1077, Jan. 2021.
\newblock arXiv: 1909.12186.

\bibitem{Burgisser2013}
P.~Bürgisser and F.~Cucker.
\newblock {\em Condition}, volume 349.
\newblock Springer Berlin Heidelberg, Berlin, Heidelberg, 2013.
\newblock Series Title: Grundlehren der mathematischen Wissenschaften
  Publication Title: Media.

\bibitem{DeLathauwer2000}
L.~De~Lathauwer, B.~De~Moor, and J.~Vandewalle.
\newblock A multilinear singular value decomposition.
\newblock {\em SIAM Journal on Matrix Analysis and Applications},
  21(4):1253--1278, 2000.

\bibitem{dedieu_approximate_1995}
J.-P. Dedieu.
\newblock Approximate solutions of numerical problems, condition number
  analysis and condition number theorem.
\newblock In {\em The {Mathematics} of {Numerical} {Analysis}}, volume~32 of
  {\em Lectures in {Applied} {Mathematics}}, pages 263--283. American
  Mathematical Society, Park City, Utah, United States, 1996.

\bibitem{dedieu_newtons_2002}
J.-P. Dedieu and M.-H. Kim.
\newblock Newton's {method} for {analytic} {systems} of {equations} with
  {constant} {rank} {derivatives}.
\newblock {\em Journal of Complexity}, 18(1):187--209, Mar. 2002.

\bibitem{demmel_condition_1987}
J.~Demmel.
\newblock On condition numbers and the distance to the nearest ill-posed
  problem.
\newblock {\em Numerische Mathematik}, 51(3):251--289, May 1987.

\bibitem{dontchevImplicitFunctionsSolution2014}
A.~L. Dontchev and R.~T. Rockafellar.
\newblock {\em Implicit {functions} and {solution} {mappings}}.
\newblock Springer {Series} in {Operations} {Research} and {Financial}
  {Engineering}. Springer New York, New York, NY, 2014.

\bibitem{degot_condition_2000}
J.~Dégot.
\newblock A condition number theorem for underdetermined polynomial systems.
\newblock {\em Mathematics of Computation}, 70(233):329--335, July 2000.

\bibitem{gohberg_mixed_1993}
I.~Gohberg and I.~Koltracht.
\newblock Mixed, {componentwise}, and {structured} {condition} {numbers}.
\newblock {\em SIAM Journal on Matrix Analysis and Applications},
  14(3):688--704, July 1993.

\bibitem{Golub2013}
G.~H. Golub and C.~F. Van~Loan.
\newblock {\em Matrix computations}, volume~3.
\newblock JHU press, Baltimore, 2013.

\bibitem{hackbusch_tensor_2012}
W.~Hackbusch.
\newblock {\em Tensor {spaces} and {numerical} {tensor} {calculus}}, volume~42
  of {\em Springer {Series} in {Computational} {Mathematics}}.
\newblock Springer Berlin Heidelberg, Berlin, Heidelberg, 2012.

\bibitem{hackbusch_interconnection_2017}
W.~Hackbusch and A.~Uschmajew.
\newblock On the interconnection between the higher-order singular values of
  real tensors.
\newblock {\em Numerische Mathematik}, 135(3):875--894, Mar. 2017.

\bibitem{halkoFindingStructureRandomness2011}
N.~Halko, P.~G. Martinsson, and J.~A. Tropp.
\newblock Finding {structure} with {randomness}: {probabilistic} {algorithms}
  for {constructing} {approximate} {matrix} {decompositions}.
\newblock {\em SIAM Review}, 53(2):217--288, Jan. 2011.

\bibitem{horn_topics_2010}
R.~A. Horn and C.~R. Johnson.
\newblock {\em Topics in matrix analysis}.
\newblock Cambridge Univ. Press, Cambridge, transferred to digital printing
  edition, 2010.

\bibitem{hornMatrixAnalysis2012}
R.~A. Horn and C.~R. Johnson.
\newblock {\em Matrix analysis}.
\newblock Cambridge University Press, Cambridge ; New York, 2nd ed edition,
  2012.

\bibitem{journee_low-rank_2010}
M.~Journée, F.~Bach, P.-A. Absil, and R.~Sepulchre.
\newblock Low-{rank} {optimization} on the {cone} of {positive} {semidefinite}
  {matrices}.
\newblock {\em SIAM Journal on Optimization}, 20(5):2327--2351, Jan. 2010.

\bibitem{kiers_three-way_2001}
H.~A.~L. Kiers and I.~Van~Mechelen.
\newblock Three-way component analysis: {principles} and illustrative
  application.
\newblock {\em Psychological Methods}, 6(1):84--110, 2001.

\bibitem{Koch2010}
O.~Koch and C.~Lubich.
\newblock Dynamical {tensor} {approximation}.
\newblock {\em SIAM Journal on Matrix Analysis and Applications},
  31(5):2360--2375, Jan. 2010.

\bibitem{Lee2013}
J.~M. Lee.
\newblock {\em Introduction to {smooth} {manifolds}}.
\newblock Springer New York, 2013.

\bibitem{Lee2018}
J.~M. Lee.
\newblock {\em Introduction to {Riemannian} manifolds}, volume 176.
\newblock Cham: Springer, 2018.

\bibitem{limGrassmannianAffineSubspaces2021}
L.-H. Lim, K.~S.-W. Wong, and K.~Ye.
\newblock The {Grassmannian} of affine subspaces.
\newblock {\em Foundations of Computational Mathematics}, 21(2):537--574, Apr.
  2021.

\bibitem{mahoneyCURMatrixDecompositions2009}
M.~W. Mahoney and P.~Drineas.
\newblock {CUR} matrix decompositions for improved data analysis.
\newblock {\em Proceedings of the National Academy of Sciences},
  106(3):697--702, Jan. 2009.

\bibitem{martin_jacobi-type_2008}
C.~D.~M. Martin and C.~F. Van~Loan.
\newblock A {Jacobi}-{type} {method} for {computing} {orthogonal} {tensor}
  {decompositions}.
\newblock {\em SIAM Journal on Matrix Analysis and Applications},
  30(3):1219--1232, Jan. 2008.

\bibitem{Oseledets2011}
I.~V. Oseledets.
\newblock Tensor-{train} {decomposition}.
\newblock {\em SIAM Journal on Scientific Computing}, 33(5):2295--2317, Jan.
  2011.

\bibitem{Rice1966}
J.~R. Rice.
\newblock A {Theory} of {Condition}.
\newblock {\em SIAM Journal on Numerical Analysis}, 3(2):287--310, June 1966.

\bibitem{shub_complexity_1993}
M.~Shub and S.~Smale.
\newblock Complexity of {Bezout}'s {theorem} {i}: {geometric} {aspects}.
\newblock {\em Journal of the American Mathematical Society}, 6(2):459, Apr.
  1993.

\bibitem{Stewart1990}
G.~Stewart and J.~Sun.
\newblock {\em Matrix perturbation theory}.
\newblock Academic Press, Inc., 1990.
\newblock Publication Title: Mathematics and Computers in Simulation ISSN:
  03784754.

\bibitem{sunPerturbationAnalysisSingular1996}
J.-g. Sun.
\newblock Perturbation analysis of singular subspaces and deflating subspaces.
\newblock {\em Numerische Mathematik}, 73(2):235--263, Apr. 1996.

\bibitem{trefethenNumericalLinearAlgebra1997}
L.~N. Trefethen and D.~Bau.
\newblock {\em Numerical linear algebra}.
\newblock Society for Industrial and Applied Mathematics, Philadelphia, 1997.

\bibitem{tucker1966some}
L.~R. Tucker.
\newblock Some mathematical notes on three-mode factor analysis.
\newblock {\em Psychometrika}, 31(3):279--311, 1966.

\bibitem{Vannieuwenhoven2017}
N.~Vannieuwenhoven.
\newblock Condition numbers for the tensor rank decomposition.
\newblock {\em Linear Algebra and its Applications}, 535:35--86, Dec. 2017.

\bibitem{vannieuwenhoven_condition_2023}
N.~Vannieuwenhoven.
\newblock The condition number of singular subspaces, revisited, Aug. 2023.
\newblock arXiv preprint.

\bibitem{Vannieuwenhoven2012}
N.~Vannieuwenhoven, R.~Vandebril, and K.~Meerbergen.
\newblock A new truncation strategy for the higher-order singular value
  decomposition.
\newblock {\em SIAM Journal on Scientific Computing}, 34(2):1027--1052, 2012.

\end{thebibliography}

\end{document}